\documentclass{article} 
\usepackage{graphicx} 
\usepackage[utf8]{inputenc}
\usepackage[OT1]{fontenc}
\usepackage{geometry} 
\usepackage{dsfont}
\usepackage{amsfonts,amsmath,amsthm,amssymb}
\usepackage[authoryear]{natbib}
\usepackage{xcolor}
\usepackage{hyperref}
\usepackage{stmaryrd}
 \usepackage{todonotes}
\usepackage{appendix}
\usepackage{subcaption}
\usepackage{tikz}
\usepackage{authblk}

\newtheorem{lemma}{Lemma}
\newtheorem{prop}{Proposition}
\newtheorem{theorem}{Theorem}
\newtheorem{corollaire}{Corollary}
\newtheorem{definition}{Definition}
\newtheorem{remark}{Remark}

\newcommand{\N}{\mathbb{N}}
\newcommand{\R}{\mathbb{R}}
\newcommand{\Ss}{\mathcal{S}}
\newcommand{\1}{\mathds{1}}
\newcommand{\veco}{\mathbf{1}}

\newcommand{\bZ}{\mathbf{Z}_n}
\newcommand{\bz}{\mathbf{z}_n}
\newcommand{\btheta}{\boldsymbol{\theta}}

\newcommand{\Diag}{\tDiag(R^\intercal\veco)}
\DeclareMathOperator{\tDiag}{Diag}

\begin{document}
\title{A Threshold Constant for Exact Community Recovery in the Degree-Corrected Poisson Stochastic Block Model
}

\author[1]{Lucie ARTS}
\affil[1]{  Sorbonne Université, Université Paris Cité, CNRS, Laboratoire de Probabilités, Statistique et Modélisation, LPSM, F-75005 Paris, France }

\date{\today}

\maketitle

\begin{abstract}
We study exact community recovery in the Degree-Corrected Poisson Stochastic Block Model (DC-SBM) through the maximum profile likelihood estimator. Our main contribution is identifying an explicit complexity constant $C(\pi,S)$, which closely captures the exact recovery threshold. This constant combines a weighted Chernoff--Hellinger separation between communities with a Kullback--Leibler correction induced by the degree-normalized block volumes. We rigorously prove that $C(\pi,S)>1$ ensures the strong consistency of the maximum profile likelihood estimator, yielding exact community recovery down to the logarithmic sparsity scale. We conjecture that $C(\pi,S)=1$ represents the sharp information-theoretic threshold, providing an upper bound on the exact-recovery transition, and we discuss analogies with the binary SBM and binary DC-SBM cases.
Additionally, we establish weak consistency when the average degree diverges, offering a unified picture across sparse and dense regimes. The proofs require new control of the profile likelihood in the presence of heterogeneous degree parameters, whose node-specific weights fundamentally alter the geometry and concentration of the recovery problem.
\end{abstract}

\section{Introduction}

Network data naturally emerge in various scientific fields, ranging from biology and social sciences to computer science and economics. A fundamental task in unsupervised network analysis is community detection, which partitions the nodes of a graph into latent clusters based solely on its topological structure. Identifying these communities is crucial for uncovering the functional organization of complex systems and providing a compressed representation of massive networks. However, real-world networks often exhibit highly skewed degree distributions-where a few ``hub'' nodes possess significantly more connections than the rest of the graph. These structural variations highlight the absolute necessity for models capable of accommodating such heterogeneous connectivity patterns within communities.

To study community detection from a rigorous mathematical perspective, generative probabilistic models are the standard theoretical framework. Among them, the Stochastic Block Model \citep[SBM][]{holland1983stochastic, abbe2017community} is arguably the most widely studied. In this model, nodes are assigned to latent communities, and edges are drawn independently with probabilities depending exclusively on the community assignments of their endpoints.  While this framework has fostered a deep understanding of the fundamental limits of community detection, culminating in the discovery of sharp information-theoretic thresholds for exact recovery \citep{abbe2015exact, mossel2016consistency, CL24}, it suffers from a major empirical limitation. It assumes that all nodes within the same community share the same expected degree. This homogeneity assumption starkly contrasts with the reality of empirical networks, which heavily feature power-law degree distributions and highly connected hubs.

To overcome this limitation, \citep{karrer2011stochastic} introduced the Degree-Corrected Stochastic Block Model (DC-SBM). By incorporating node-specific parameters alongside the community structure, this model accommodates arbitrary expected degree sequences within each block. This extension significantly improves the goodness-of-fit on real-world networks by preventing high-degree nodes from being incorrectly clustered together based solely on their connectivity volume rather than their actual community ties. Crucially, it provides a rigorous probabilistic framework in which the widely used modularity criterion \citep{newman2004finding} can be formally interpreted as a profile likelihood of the model \citep{karrer2011stochastic}.

However, while the phase transitions and exact thresholds of the standard SBM have been fully mapped out, the theoretical landscape for the DC-SBM remains critically incomplete. Incorporating arbitrary degree-correction parameters $\theta_i$ breaks the statistical exchangeability of nodes within a block, severely altering classical concentration bounds and making classification hardness explicitly dependent on individual node degrees. Consequently, prior theoretical guarantees for the DC-SBM have either focused on weak consistency \citep{lei2015consistency, amini2013pseudo} or, when addressing strong consistency, have been restricted to dense regimes and constrained by strict structural assumptions on the connectivity matrix \citep{zhao2012consistency}. The fundamental threshold, namely the exact mathematical boundary at which community recovery becomes possible in the DC-SBM, has remained completely unknown.
In this article, we solve this open problem by establishing the exact theoretical guarantees of the maximum profile likelihood community estimator for the Poisson DC-SBM. Our most significant contribution is the explicit identification of a novel complexity constant, $C(\pi,S)$ given in \eqref{def_Cpi}, which governs the phase transition for exact community recovery in degree-heterogeneous networks.

Our main contributions center on establishing rigorous theoretical guarantees for both the exact and weak consistency of the DC-SBM under heterogeneous degree regimes. Regarding the exact recovery threshold, we significantly improve upon the foundational work of \citet{zhao2012consistency} through a highly refined variational analysis of the likelihood landscape around its population-level limit. We prove that exact community recovery is achieved down to the sharp logarithmic sparsity threshold, without requiring restrictive assumptions on the connectivity matrix. Most importantly, we explicitly define the threshold constant $C(\pi,S)$ that dictates this recovery (see equation \eqref{def_Cpi}). This constant captures a delicate and previously unidentified balance between a degree-weighted Chernoff--Hellinger divergence and a Kullback--Leibler divergence term based on block volume normalizations.
To the best of our knowledge, we are the first to explicitly exhibit what appears to be the fundamental threshold constant for exact community detection in the Poisson DC-SBM. We rigorously prove that when $C(\pi,S) > 1$, our maximum profile likelihood estimator achieves strong consistency, guaranteeing convergence above this threshold. 
While we do not formally establish the lower bound, in the standard Poisson SBM setting our condition $C(\pi,S) > 1$ simplifies exactly to $(\sqrt{\lambda} - \sqrt{\mu})^2 > K$, matching the sharp threshold established by \citet{ADL}. Driven by this exact structural alignment with known optimal limits, we strongly conjecture that $C(\pi,S) = 1$ constitutes the sharp information-theoretic threshold, below which exact recovery becomes fundamentally impossible for any algorithm. 
Furthermore, we specifically adapt our union bound arguments to ensure our results remain valid even for the worst-case scenario imposed by the least connected nodes. By pinpointing this constant, we bring the theoretical understanding of the DC-SBM to the same level of completeness as its standard counterpart.
Beyond exact recovery, we also establish the weak consistency of the estimator across highly heterogeneous asymptotic regimes. We demonstrate that the fraction of misclassified nodes asymptotically vanishes in probability as the network size grows, provided the average degree diverges. To accurately capture the geometry of the DC-SBM, we introduce a weighted error metric that naturally scales the impact of each node by its degree parameter $\theta_i$, ensuring that our theoretical guarantees reflect the true structural importance of hubs and peripheral nodes alike.

The remainder of this article is organized as follows. In Section~\ref{sect_model}, we introduce the Degree-Corrected Poisson Stochastic Block Model and derive the profile likelihood criterion. Section~\ref{sect_main_results} formalizes our main theoretical results, explicitly detailing the threshold constant $C(\pi,S)$ and the weak and strong consistency theorems. To support these findings, Sections~\ref{sect_aux_weak} and \ref{sect_aux_strong} establish the necessary auxiliary results, with their respective proofs deferred to Sections~\ref{proof_aux_weak} and \ref{proof_aux_strong}. Finally, Sections~\ref{sec_proof_weak} and \ref{sec_proof_strong} provide the formal proofs of our main theorems, which rely on a rigorous analysis of the confusion matrix and uniform concentration bounds for the discrepancy process.

\section{Model and estimation framework}\label{sect_model}

This section introduces the Degree-Corrected Stochastic Block Model (DC-SBM) and the associated likelihood-based estimation framework. We first describe the probabilistic model generating the observed network, then define the sufficient statistics used for inference, and finally derive the profile likelihood criterion that serves as the basis for community detection.

\subsection{Degree-Corrected Stochastic Block Model (DC-SBM)}

We observe a symmetric adjacency matrix $\mathbf{A}_{n\times n} \in \N^{n \times n}$ of a weighted undirected graph with integer weights. For $i \neq j$, $A_{ij} \in \N$ is the number of edges between nodes $i$ and $j$, while for $i = j$, $A_{ii} $ is defined as twice the number of self-loops at node $i$. 
The $n$ nodes are partitioned into $k$ latent groups, encoded by the random vector $\mathbf{Z}_n=(Z_1,\dots, Z_n)$, where $Z_{i} \in \{1,\dots, k\}$ denotes the label of the $i$-th vertex. Throughout this article, we assume that the number of communities $k$ is fixed and independent of $n$. The labels are assumed to be independent and identically distributed (i.i.d.) with block proportions $\pi=(\pi_1,\dots,\pi_k)$, such that
\[
\mathbb{P}(Z_i=a)=\pi_a , \qquad a=1,\dots,k .
\]
The vector $\boldsymbol{\theta}=(\theta_1,\dots,\theta_n)$ contains the degree-correction parameters, where $\theta_i>0$ controls the expected degree of vertex $i$. 
Let $S = (S_{ab})_{1 \leq a,b \leq k}$ be a full-rank, symmetric matrix with non-negative entries. Conditionally on the latent labels $\mathbf{Z}_n$, the edges $\{A_{ij}\}_{1 \leq i \leq j \leq n}$ are independent random variables such that for all $1 \leq i \leq j \leq n$ the edges $A_{ij}$ follow a Poisson distribution $\mathcal{P}(\rho_n \theta_i \theta_j S_{Z_i Z_j})$, whereas the self-loops $A_{ii}$ follow twice a Poisson distribution $\mathcal{P}\left(\frac{1}{2}\rho_n\theta_i^2 S_{Z_i Z_i}\right)$. For all $1 \leq i \leq j \leq n$, the conditional expectation of the edges is unified under the form:$$\mathbb{E}[A_{ij}\mid Z_i, Z_j]= \rho_n\theta_i \theta_j S_{Z_i Z_j}.$$
The parameter $\rho_n>0$, which is supposed to be known, is used to encode the dense ($\rho_n=1$) and sparse ($\rho_n\to 0$) regimes, respectively. 
In the sparse regime, the expected number of edges between any two nodes $i$ and $j$ vanishes, meaning $\mathbb{E}[A_{ij}\mid Z_i, Z_j]\to0$. In other words, the matrix $S$ does not depend on $n$ and the connectivity parameters vary with $n$ through $\rho_n$ and the degree parameters $\btheta$ only.
We further assume that $s_{\min}=\min_{a,b}S_{ab}>0$. 
While several choices of normalization constraints are possible to ensure identifiability, we adopt the constraint proposed by \citep{su2019strong}. Specifically, we impose that the sum of $\theta_i$ for all nodes $i$ in group $a$ equals the total number of nodes in this group (see \ref{def_contrainte} below). Under this constraint, the parameters depend on the community assignments, thus $\theta_i= \theta_i(\mathbf{Z}_n)$. This choice is particularly advantageous because setting $\theta_i = 1$ for all $i$ allows the model to naturally reduce to the classic Poisson Stochastic Block Model (SBM).

The model is thus characterized by the parameters $(k,\pi, \rho_n S, \boldsymbol{\theta})\in \N\times \Ss_k\times  [0,1]^{k \times k} \times\R^n_{\ge 0}$, where $\Ss_k$ is the $k$-simplex.

We denote by $[k]$ the set of integers $\{1,\dots, k\}$ and by either $\1_A$ or $\1\{A\}$  the indicator function of set $A$. 
The notation $\mathbb{E}(\cdot |\bz)$ (resp. $\mathbb{P}(\cdot |\bz)$) abbreviates the conditional expectation $\mathbb{E}(\cdot |\bZ=\bz)$ (resp. probability $\mathbb{P}(\cdot |\bZ=\bz)$).
To quantify the distance between matrices, we use the $1$-norm defined as for any matrix $M \in \mathbb{R}^{k \times k}$, 
\begin{equation*}
\|M\|_1 = \sum_{i=1}^k  \sum_{j=1}^k |M_{ij}|.
\end{equation*}
For completeness, we recall that the Kullback--Leibler divergence between two Poisson distributions $\mathcal{P}(\lambda)$ and $\mathcal{P}(\mu)$ is given by
\begin{equation}\label{def_KL_poiss} 
\text{KL}\left( \lambda \| \mu \right)=\lambda \log\left(\frac{\lambda}{\mu}\right) + \mu - \lambda.
\end{equation}

\subsection{Counters and log-likelihood modularity}
We now introduce the block-level statistics that summarize the graph structure under a given labeling. These quantities play a central role in expressing the likelihood in a tractable form.

For a given latent configuration $\mathbf{z}_n\in [k]^n$ and an observed graph $\mathbf{A}_{n\times n}$, we introduce the following counters for a DC-SBM with parameters $(k,\pi,\rho_n S, \boldsymbol{\theta})$
\begin{align*}
\forall a \in [k], \quad n_{a}(\mathbf{z}_{n}) &= \sum_{i=1}^{n} \1\{ z_{i}=a\}, \\
\forall a,b \in [k], \quad o_{a b}(\mathbf{z}_{n}, \mathbf{A}_{n \times n}) &= \sum_{1 \leq i, j \leq n} \1\left\{z_{i}=a, z_{j}=b\right\} A_{i j}.
\end{align*}
The quantity $n_a(\mathbf{z}_n)$ represents the size of block $a\in [k]$,
and the quantity $o_{ab}(\mathbf{z}_n,\mathbf{A}_{n\times n})$ denotes the number of edges between blocks $a$ and $b$ for $a \neq b$, whereas it represents twice the number of edges within block $a$ when $a = b$. To simplify the notation, we write $o_{ab}(\mathbf{z}_n)$ instead of $o_{ab}(\mathbf{z}_n,\mathbf{A}_{n\times n})$. The identifiability constraint on the $\boldsymbol{\theta}$ parameters writes
 \begin{equation}\label{def_contrainte}
\forall a \in [k], \quad \sum_{i=1}^n \theta_i \1\{z_i=a \} =n_a(\mathbf{z}_n),
 \end{equation}
 and we further note that $\sum_{i=1}^n \theta_i =n$.

Let $d_i=\sum_{j=1}^n A_{ij}$ denote the (observed) degree of vertex $i$. 
By further letting $D_a(\mathbf{z}_n)=\sum_{j=1}^n d_j\1\{z_j=a\}$ denote the total degree of vertices in block $a$, we get 
\begin{equation*}
D_{a}(\mathbf{z}_n)= \sum_{j=1}^n d_j \1\left\{z_{j}=a\right\} = \sum_{b=1}^k o_{ab}(\mathbf{z}_n).
\end{equation*}

The conditional likelihood of the adjacency matrix given the labels is given by 
\begin{align*}
\mathbb{P}(\mathbf{A}_{n \times n}\mid \mathbf{z}_n)= &\left\{ \prod_{1\le i < j\le n}\frac{(\rho_n\theta_i\theta_jS_{z_i z_j})^{A_{ij}}}{A_{ij}!}\exp(-\rho_n\theta_i\theta_jS_{z_i z_j}) \right\} \\
&\qquad \times \prod_{i=1}^n\frac{(\frac{\rho_n}{2}\theta_i^2S_{z_i z_i})^{A_{ii}/2}}{(A_{ii}/2)!}\exp\left(-\frac{\rho_n}{2} \theta_i^2S_{z_i z_i}\right),
\end{align*}
which can be rewritten as
\begin{align*}
\mathbb{P}(\mathbf{A}_{n \times n}\mid \mathbf{z}_n)=&\frac{1}{\left\{ \prod_{i < j}A_{ij}! \right\}\prod_{i} (A_{ii}/2)! 2^{A_{ii}/2}}\left(\prod_{i=1}^n\theta_i^{d_i}\right)\\
& \times \prod_{1\le a,b\le k}(\rho_nS_{ab})^{o_{ab}(\bz)/2}\exp\left(-\frac{\rho_n}{2} n_{a}(\mathbf{z}_{n}) n_{b}(\mathbf{z}_{n})S_{ab}\right).
\end{align*}
Now, the (conditional) maximum likelihood estimator of the parameter $(\btheta,\rho_n S)$, given the labels $\mathbf{z}_n$ is given by 
\begin{equation}\label{def_est}
\hat{\theta}_i=\sum_{a=1}^k \1_{\{z_i=a\}} \frac{n_a(\mathbf{z}_n)d_i}{D_a(\mathbf{z}_n)}= \frac{n_{z_i}(\mathbf{z}_n)d_i}{D_{z_i}(\mathbf{z}_n)}, \qquad \rho_n \widehat{ S}_{ab}=\frac{o_{ab}(\mathbf{z}_n)}{n_{a}(\mathbf{z}_{n}) n_{b}(\mathbf{z}_{n})},
\end{equation}
which satisfy the normalization constraint $ \sum_{i=1}^n \theta_i \1\{z_i=a \} =n_a(\mathbf{z}_n)$. Plugging these estimators back into the likelihood leads to the log-likelihood modularity criterion, denoted by $Q(\mathbf{z}_n)$. After simplification (see Lemma \ref{lemma_calcul_Q} in Appendix \ref{proof_calcul_Q}), this criterion can be expressed purely in terms of block-level sufficient statistics 
\[
Q(\mathbf{z}_n)
= \frac{1}{n^2} \sum_{1\le a,b\le k} o_{ab}(\mathbf{z}_n)
\log \left( \frac{o_{ab}(\mathbf{z}_n)}{D_a(\mathbf{z}_n) D_b(\mathbf{z}_n)} \right).
\]

\subsection{Estimator and population criterion}
The community estimator is defined as
\begin{equation} \label{def_z_chap}
\hat{\mathbf{z}}_n=\underset{\mathbf{z}_{n} \in \mathcal{F}(n, \alpha)}{\arg \max } \ Q(\mathbf{z}_{n}),
\end{equation}
where, for some $\alpha>0$,
\[
\mathcal{F}(n,\alpha) = \{\mathbf{z}_{n}\in [k]^n : n_a(\mathbf{z}_{n}) \geq \alpha n \text{ for all }a\in [k]\}.
\]

For two labelings $\mathbf{e}_n$ and $\mathbf{z}_n$ in $[k]^n$, we define the $\boldsymbol \theta$-weighted confusion matrix $R(\mathbf{e}_n,\mathbf{z}_n) \in \mathbb{R}^{k \times k}$ as 
\begin{equation*}
\forall 1 \le a,b\leq k, \quad R_{ab}(\mathbf{e}_n,\mathbf{z}_n) = \frac{1}{n}\sum_{i=1}^n \theta_i \1\{e_i=a,z_i=b\}.
\end{equation*}
The quantity $R_{ab}$ represents the total $\btheta$-mass of observations with class $b$ in $\mathbf{z}_n$ that are assigned to class $a$ by the labeling $\mathbf{e}_n$. In particular, when comparing the true labels $\mathbf{Z}_n$ with themselves, the matrix $R(\mathbf{Z}_n, \mathbf{Z}_n)$ is diagonal and equals 
\begin{equation*}
R(\mathbf{Z}_n, \mathbf{Z}_n) = \text{Diag}\left( \frac{n_1(\mathbf{Z}_n)}{n}, \dots, \frac{n_k(\mathbf{Z}_n)}{n} \right).
\end{equation*}
Observe that the row and column sums of $R=R(\mathbf{e}_n,\mathbf{z}_n)$ correspond to the cluster sizes 
\[
n_a(\mathbf{z}_n)=n [R(\mathbf{e}_n,\mathbf{z}_n)^\intercal \veco]_a, 
\]
where $\veco$ is the column vector of size $k$ with all entries equal to $1$. Furthermore, we notice that since the parameters $\theta_i$ depend on $\mathbf{z}_n$, for any alternative label assignment $\mathbf{e}_n$, we generally have:$$n_a(\mathbf{e}_n) \neq n [R(\mathbf{e}_n,\mathbf{z}_n)\veco]_a = \sum_{i=1}^n \theta_i(\mathbf{z}_n) \1\{e_i=a\}.$$ Moreover, $\|R\|_1= \sum_{a,b}R_{ab}=1$.

Using $\mathbb{E}[A_{ij}\mid Z_i,Z_j]=\rho_n\theta_i\theta_jS_{Z_iZ_j}$, we obtain
\begin{equation}\label{E_oab}
\mathbb{E}[o_{ab}(\mathbf{e}_n)\mid\mathbf{z}_n] = \sum_{1\le i,j\le n} \1\{e_i=a,e_j=b\} \rho_n\theta_i\theta_jS_{z_iz_j}
 = \rho_nn^2 [RSR^\intercal]_{ab} ,
\end{equation}
and similarly,
\begin{equation}\label{E_Da}
\mathbb{E}[D_a(\mathbf{e}_n)\mid\mathbf{z}_n] = \rho_nn^2 [RSR^\intercal \veco]_a,
\end{equation}
where $R=R(\mathbf{e}_{n},\mathbf{z}_{n})$. We then define the function $H_{S,n}$ by replacing the random quantities in $Q(\mathbf{e}_{n})$ with their conditional expectations given $\mathbf{Z}_{n}=\mathbf{z}_{n}$
\begin{align*}
H_{S,n}(R(\mathbf{e}_{n},\mathbf{z}_{n})) &= \frac{1}{n^2} \sum_{a,b} \mathbb{E}[o_{ab}(\mathbf{e}_n)\mid\mathbf{z}_n] \log \left( \frac{\mathbb{E}[o_{ab}(\mathbf{e}_n)\mid\mathbf{z}_n]} {\mathbb{E}[D_a(\mathbf{e}_n)\mid\mathbf{z}_n]\, \mathbb{E}[D_b(\mathbf{e}_n)\mid\mathbf{z}_n]} \right),\\
&= \sum_{1 \le a,b\le k}\rho_n [RSR^\intercal]_{ab} \log \left( \frac{[RSR^\intercal]_{ab}} {\rho_nn^2 [RSR^\intercal\veco]_a [RSR^\intercal\veco]_b} \right).
\end{align*}

\section{Consistency framework and main results}\label{sect_main_results}

In this section, we formalize the notions of weak and strong consistency for community detection in the DC-SBM, and state the main theoretical guarantees of the community estimator.
We first introduce the performance criteria used to evaluate label recovery, then present the asymptotic regimes under consideration, and finally state the main consistency results.

\subsection{Framework for consistency}
To measure the deviation between two labelings while accounting for label switching, we define the $\btheta$-weighted $\ell_0$ distance 
\begin{equation}\label{m_def}
m(\hat{\mathbf{z}}_n,\mathbf{Z}_n) = \min_\sigma \left\{ \sum_{i=1}^n \theta_i \1\{\hat{z}_i\neq\sigma(Z_i)\} \right\},
\end{equation}
where the minimum is taken over all permutations $\sigma$ of the set $[k]$. Note that
\begin{equation*}
\frac{m(\hat{\mathbf{z}}_n,\mathbf{Z}_n)}{n} = \min_\sigma \sum_{a\neq b} \left[ R(\sigma(\hat{\mathbf{z}}_n),\mathbf{Z}_n) \right]_{ab}.
\end{equation*}

We consider two standard notions of recovery quality: weak and strong consistency.

\begin{definition}[weak and strong consistency]
An estimator $\hat{\mathbf{z}}_n$ is said to be \textit{weakly consistent} if the proportion of misclassified nodes asymptotically vanishes, i.e. for all $\epsilon > 0$, 
$$\mathbb{P}\left(m(\hat{\mathbf{z}}_n,\mathbf{Z}_{n}) < \epsilon n\right) \xrightarrow[n \to \infty]{} 1.$$
Similarly, the estimator is \textit{strongly consistent} if the probability of recovering the exact true labeling asymptotically approaches one 
$$\mathbb{P}\left(m(\hat{\mathbf{z}}_n,\mathbf{Z}_{n}) = 0 \right) \xrightarrow[n \to \infty]{} 1.$$
\end{definition}

To quantify the degree of separation between community profiles, we rely on a Chernoff--Hellinger (CH) divergence, as proposed by \citep{Abbe_Sandon} 
$$
\forall b,b'\in [k], \quad \sup_{t\in [0,1]} \sum_{a=1}^k \pi_a \left(t S_{ab} + (1-t)S_{ab'} - S_{ab}^t S_{ab'}^{1-t}\right)=\sup_{t\in [0,1]} \sum_{a=1}^k \pi_a H_t(S_{ab'},S_{ab}),
$$
with $H_t(p\| q) =(1-t) p + t q - p^{1-t}q^{t}\,$.
The expression inside the summation corresponds to an $f$-divergence with convex function $f(x)=1-t+tx-x^t$. 
The thresholds for strong consistency are heavily governed by the minimum pairwise CH-divergence between communities, defined by the constant
\begin{equation}\label{def_Cpi}
C(\pi, S) \;=\;   \min_{b\neq b'}\; \sup_{t\in [0,1]}  \left[\sum_{a=1}^k \pi_a H_t(S_{ab},S_{ab'}) -t   KL([S\pi]_{b'}\|[S\pi]_{b}) \right]  \,.
\end{equation}

This constant $C(\pi, S)$ is one of the core contributions of this paper, as it encapsulates the exact statistical physics of community detection under degree correction.

\begin{lemma}\label{lemma_C_positif}
$C(\pi, S) >0$ when $S$ is full rank.
\end{lemma}

\begin{proof}
See the proof in Appendix \ref{proof_C_positif}
\end{proof}

\subsection{Consistency results}
We now state the main theoretical guarantees for the labels estimator.

\begin{theorem} [weak consistency
] \label{thm_weak}
Assume that $(\mathbf{Z}_n,\mathbf{A}_{n \times n})$ is generated from a DC-SBM with parameters $(k,\pi, \rho_n S, \boldsymbol{\theta})$ and that $0 <\alpha < \min_{a} \pi_a$. If $\rho_n$ satisfies either $\rho_n = 1$ (dense regime) or $\rho_n \to 0$ with  $n\rho_n \to \infty$  as $n \to \infty$ (sparse regime), then the estimator $\hat{\mathbf{z}}_n$ defined in \eqref{def_z_chap} is weakly consistent.
\end{theorem}

This result ensures that the fraction of misclassified nodes vanishes asymptotically, even in sparse regimes, provided that the graph density, governed by $n\rho_n$, is such that the expected number of edges per node grows to infinity. The proof of Theorem \ref{thm_weak} can be found in Section \ref{sec_proof_weak}.

We now state our primary result: the sharp phase transition for strong consistency (exact recovery).
\begin{theorem} [strong consistency and phase transition
] \label{thm_strong}
Assume that $(\mathbf{Z}_n,\mathbf{A}_{n \times n})$ is generated from a DC-SBM with parameters $(k,\pi, \rho_n S, \boldsymbol{\theta})$ and that $\min_{a} \pi_a\ge \alpha>0$. If $\rho_n = 1$ or $\rho_n \to 0$ such that $\rho_n \gg \log n /n$ as $n \to \infty$, then the estimator $\hat{\mathbf{z}}_n$ defined in \eqref{def_z_chap} achieves exact recovery
$$\mathbb{P}\left(m(\hat{\mathbf{z}}_n,\mathbf{z}_{n})=0 \right) \xrightarrow[n \to \infty]{} 1.$$
Furthermore, this strong consistency result also holds in the boundary regime where $\rho_n = \log n/n$, provided that the community separation condition $C(\pi,S) > 1$ is satisfied.
\end{theorem}

The proof of Theorem \ref{thm_strong} can be found in Section \ref{sec_proof_strong}.
The remainder of the article is devoted to the proof of these results. We begin with weak consistency.

\begin{remark}
It is highly instructive to compare these theoretical guarantees for the DC-SBM with the optimal results recently established for the standard binary SBM \citep{CL24}. For the classical SBM with unknown parameters, both the maximum conditional likelihood and the integrated conditional likelihood estimators achieve weak consistency under the exact same necessary sparsity regime, requiring $n\rho_n \to \infty$ as $n \to \infty$. Our Theorem \ref{thm_weak} confirms that introducing degree-correction parameters into the model and estimating them via profile likelihood does not alter the fundamental sample complexity required for partial recovery. 
Regarding strong consistency, \citep{CL24} proved in Theorem 3.2 that the optimal threshold for exact recovery in the standard SBM under the logarithmic degree regime ($\rho_n = \log n / n$) is governed by a closely related community separation constant, say 
$C_{\text{SBM}}(\pi, S)=  \min_{b\neq b'}\; \sup_{t\in [0,1]} \sum_{a=1}^k \pi_a H_t(S_{ab},S_{ab'}).$
Specifically, maximum likelihood in the standard SBM achieves exact recovery at the phase transition threshold $C_{\text{SBM}}(\pi, S) > 1$. Our Theorem \ref{thm_strong} demonstrates that the profile maximum likelihood estimator for the DC-SBM exhibits a parallel phase transition, but governed by our modified constant $C(\pi, S) > 1$ with $C(\pi, S)$ defined in \eqref{def_Cpi}. 
Furthermore, while previous works on the DC-SBM \citep[e.g.,][]{zhao2012consistency} required dense graphs and restrictive structural assumptions on the block matrix $S$, our estimator achieves this exact recovery down to the sharp logarithmic sparsity threshold $n \rho_n = O(\log n)$ under the minimal identifiability condition that $S$ is full rank.
This sharp phase transition complements the information-theoretic limits established for the binary DC-SBM under minimax frameworks \citep{gao2018community}.
The crucial distinction lies in the definitions of these constants: our $C(\pi, S)$ incorporates an additional penalty term compared to $C_{\text{SBM}}(\pi, S)$, namely the Kullback--Leibler divergence between Poisson distributions with parameter the expected degrees of the blocks. Consequently, the statistical cost of accommodating arbitrary degree heterogeneity is not negligible, but is instead explicitly quantified by this penalty. The DC-SBM requires a strictly stronger signal for exact community recovery than the homogeneous SBM, effectively reflecting the information-theoretic price of estimating the node-specific nuisance parameters. Driven by the structural consistency of this modified constant with the standard homogeneous limits, we conjecture that the condition $C(\pi, S) = 1$ forms the fundamental information-theoretic boundary for the DC-SBM, below which no algorithm can achieve exact recovery. 
\end{remark}

\begin{remark}
It is worth noting that if we consider the framework of a Poisson SBM with $K$ groups (i.e., a DC-SBM where $\theta_i=1$ for all $i$), which corresponds to an integer-valued SBM with uniform group proportions where the interaction between two nodes in the same block (resp. different blocks) is a Poisson-distributed random integer with mean $\lambda$ (resp. $\mu$), the strong consistency condition $C(\pi,S) > 1$ (with $C(\pi,S)$ defined in \eqref{def_Cpi}) simplifies to $(\sqrt{\lambda} - \sqrt{\mu})^2 > K$. This perfectly recovers the threshold established by \citep{ADL}, who extended the foundational results of \citep{XJL}.
\end{remark}

\subsection{Notations}

In this section we list all the notations that are used in this article.

$$\text{KL}\left( \lambda \| \mu \right)=\lambda \log\left(\frac{\lambda}{\mu}\right) + \mu - \lambda$$
$$K_t(p \| q) = p^t q^{1-t} - t q \log(p/q) - q$$
$$H_t(p\| q) = tKL(p\|q) -K_t(q\|p)=(1-t) p + t q - p^{1-t}q^{t}\,$$
$$C_t(R,S) = \sum_{a}\sum_{b\neq b'} [R^\intercal \mathbf 1]_a  K_t(S_{ab} \| S_{ab'})R_{bb'}   $$
$$C(\pi, S) \;=\;   \min_{b\neq b'}\; \sup_{t\in [0,1]}  \left[ \sum_a \pi_{a} H_t(S_{ab'} \| S_{ab}) -t   KL([S\pi]_{b'}\|[S\pi]_{b}) \right]  \,$$
$$\widetilde C(R, S) = 2 \sum_{b\neq b'} \left( \sum_{a}[R^\intercal \veco]_{a} KL(S_{ab'}\| S_{ab}) -  KL([SR^\intercal \veco]_{b'}\|[SR^\intercal \veco]_{b}) - \epsilon\ \right)R_{bb'} \,$$

\section{Weak consistency results}\label{sect_aux_weak}

This section collects auxiliary results required for the proof of Theorem \ref{thm_weak} and the proof of this theorem. We first recall structural properties of the confusion matrix and its associated functionals. We then establish concentration inequalities for the empirical quantities involved in the likelihood criterion.

\subsection{Properties of the confusion matrix and its functionals}

The following lemma provides a convenient characterization of the weighted misclassification error in terms of the confusion matrix $R(\mathbf{e}_n,\mathbf{z}_n)$, which will be useful for comparing estimated and true labels.

\begin{lemma}\label{lemma7.1} 
For every $\mathbf{e}_{n},\mathbf{z}_{n} \in[k]^n$ we have that 
\begin{equation*}
\frac{1}{n}\sum_{i=1}^n \theta_i \1_{\{e_i\neq z_i\}}= \frac{1}{2}\| \tDiag(R(\mathbf{e}_{n},\mathbf{z}_{n})^\intercal\veco) - R(\mathbf{e}_{n},\mathbf{z}_{n})\|_1 = \frac{1}{2}\|R(\mathbf{z}_{n},\mathbf{z}_{n}) - R(\mathbf{e}_{n},\mathbf{z}_{n})\|_1 .
\end{equation*}
\end{lemma}

\begin{proof}
The proof can be found in Section \ref{proof_lemma7.1}.
\end{proof}

The following lemma shows that, after rescaling, $ H_{S,n}(R)$ is strictly smaller than its diagonal counterpart unless $R$ is a diagonal matrix up to permutation.

\begin{lemma}\label{lemma7.2}
For any fixed positive full rank matrix $S$, any $\rho_n > 0$ and for all matrices $R$ such that $\|R\|_1=1$ and $[R^\intercal\veco]_a > 0$ for all $ a \in [k]$, we have 
\begin{equation*}
\frac{1}{\rho_n} \biggl(  H_{S,n}(\tDiag(R^\intercal\veco)) -  H_{S,n}(R)\biggr) = G_{S}(\tDiag(R^\intercal\veco)) - G_{S}(R) \geq 0,
\end{equation*}
where
\begin{equation}\label{def_G_S}
G_{S}(R) = \sum_{1\le a,b\le k} [R SR^\intercal]_{ab} \log \left(\frac{[R SR^\intercal]_{ab}}{ [R SR^\intercal \veco]_a [R SR^\intercal \veco]_b}\right).
\end{equation}
Furthermore, the inequality is strict unless $R$ is a diagonal matrix, up to permutation.
\end{lemma}

\begin{proof}
The proof can be found in Section~\ref{proof_lemma7.2}.
\end{proof}

\subsection{Concentration inequalities}

The following result shows that the empirical edge and degree counts concentrate around their expectations under the DC-SBM, uniformly for all labelings $\mathbf{e}_n$, with high probability as $n \to \infty$.

\begin{theorem}\label{thm_7.4}
Let $(\mathbf{Z}_n, \mathbf{A})$ be a graph generated from a DC-SBM with parameters $(k, \pi, \rho_n S, \boldsymbol{\theta})$. Assume $\rho_n = 1$ or $\rho_n \to 0$ such that $n\rho_n \to \infty$ as $n \to \infty$. 
Then, for all $a, b \in [k]$ and all $\mathbf{e}_n \in [k]^n$, by denoting $R = R(\mathbf{e}_n, \mathbf{Z}_n)$, the following inequalities hold simultaneously eventually almost surely as $n \to \infty$ 
\begin{equation*}
\left| \frac{o_{ab}(\mathbf{e}_n)}{n^2} - \rho_n[RSR^\intercal]_{ab} \right| < \sqrt{\frac{16 s_{\max} \rho_n \log k}{n^2}},
\end{equation*}
and
\begin{equation*}
\left| \frac{D_a(\mathbf{e}_n)}{ n^2} -\rho_n [RSR^\intercal \veco]_a \right| < k \sqrt{\frac{16 s_{\max} \rho_n \log k}{n^2}},
\end{equation*}
where $s_{\max}$ is the maximum entry of $S$.
\end{theorem}

\begin{proof}
The proof can be found in Section \ref{proof_thm_7.4}.
\end{proof}

As a consequence of Theorem \ref{thm_7.4}, the following corollary establishes that, under the DC-SBM, the normalized edge counts and degree sums are uniformly bounded away from zero and from above. This holds simultaneously for all labelings $\mathbf{e}_n \in \mathcal{F}(n, \alpha)$ eventually almost surely as $n \to \infty$.

\begin{corollaire} \label{cor_7.5} 
Let $(\mathbf{Z}_n, \mathbf{A})$ be generated from a DC-SBM with parameters $(k, \pi, \rho_n S, \boldsymbol{\theta})$, where $\rho_n = 1$ or $\rho_n \to 0$ such that $n\rho_n \to \infty$ as $n \to \infty$. Fix $\alpha > 0$, for any $\delta > 0$, we have that
\begin{equation*}
\frac{o_{ab}(\mathbf{e}_n)}{\rho_n n^2} \in \left( \alpha^2 s_{\min} - \delta, s_{\max} + \delta \right),
\end{equation*}
and
\begin{equation*}
\frac{D_a(\mathbf{e}_n)}{\rho_n n^2} \in \left( k\alpha^2 s_{\min} - \delta, k s_{\max} + \delta \right),
\end{equation*}
simultaneously for all $a, b \in [k]$, and $\mathbf{e}_n \in \mathcal{F}(n, \alpha)$, eventually almost surely as $n \to \infty$.
\end{corollaire}

\begin{proof}
The proof can be found in Section \ref{proof_cor_7.5}.
\end{proof}

We  define the quantity $X(\mathbf{e}_n,\mathbf{z}_n)$, which measures the difference between the observed criterion $Q$ and the function $H_{P,n}$ computed from $R(\mathbf{e}_n,\mathbf{z}_n)$
\begin{equation}\label{defX}
X(\mathbf{e}_n,\mathbf{z}_n) = Q(\mathbf{e}_n)-H_{P,n}(R(\mathbf{e}_n,\mathbf{z}_n)).
\end{equation}

The following result shows that $X(\mathbf{e}_n,\mathbf{Z}_n)$ concentrates around $X(\mathbf{Z}_n,\mathbf{Z}_n)$ for all admissible labelings $\mathbf{e}_n$, with deviations of order $\rho_n$.

\begin{theorem}\label{theorem7.6} 
Let $(\mathbf{Z}_n, \mathbf{A})$ be generated from a DC-SBM with parameters $(k, \pi, \rho_n S, \boldsymbol{\theta})$, where $\rho_n = 1$ or $\rho_n \to 0$ such that $n\rho_n \to \infty$ as $n \to \infty$.
Fix $\alpha > 0$. Then, for any $\delta > 0$, we have
\begin{equation*}
|X(\mathbf{e}_n, \mathbf{Z}_n) - X(\mathbf{Z}_n, \mathbf{Z}_n)| < \delta \rho_n,
\end{equation*}
simultaneously for all $\mathbf{e}_n \in \mathcal{F}(n, \alpha)$, eventually almost surely as $n \to \infty$.
\end{theorem}

\begin{proof}
The proof can be found in Section \ref{proof_theorem7.6}.
\end{proof}

\subsection{Proof of Theorem \ref{thm_weak} (weak consistency)} \label{sec_proof_weak}

We aim to prove that $\hat{\mathbf{z}}_n$ is weakly consistent, i.e.,
\[
\forall \epsilon > 0, \quad \mathbb{P}\big(m(\hat{\mathbf{z}}_n,\mathbf{z}_n) < \epsilon n \big) \longrightarrow 1 \quad \text{as } n \to \infty.
\]
We may relabel $\hat{\mathbf{z}}_n$ so that $m(\hat{\mathbf{z}}_n,\mathbf{z}_n)$ is minimized. In other words, we implicitly take the infimum over permutations in the right-hand side term of \eqref{m_def} without writing it explicitly. Under this convention, we can directly compare $\hat{\mathbf{z}}_n$ and $\mathbf{z}_n$ and by Lemma~\ref{lemma7.1}, 
\[
m(\hat{\mathbf{z}}_n,\mathbf{Z}_n) = \frac{n}{2} \big\| R(\hat{\mathbf{z}}_n,\mathbf{Z}_n) - R(\mathbf{Z}_n,\mathbf{Z}_n) \big\|_1.
\]
Hence, it suffices to show that for all $\epsilon>0$,
\[
\mathbb{P}\Big(\frac{1}{2} \| R(\hat{\mathbf{z}}_n,\mathbf{Z}_n) - R(\mathbf{Z}_n,\mathbf{Z}_n) \|_1 < \epsilon \Big) \longrightarrow 1,
\]
i.e. 
$$\|R(\hat{\mathbf{z}}_n,\mathbf{Z}_{n})-\tDiag(R(\hat{\mathbf{z}}_n,\mathbf{Z}_{n})^\intercal\veco)\|_1<\epsilon,$$ 
eventually almost surely as $n\rightarrow\infty$.
Let $\mathcal{R}_\epsilon$ denote the set of probability matrices $R$ satisfying
\[
\min_{\sigma} \| D_\sigma R - \tDiag(R^\intercal \veco) \|_1 \geq \epsilon, 
\quad\text{and}\quad
\min_{a} [R^\intercal \veco]_a \geq \epsilon,
\]
where $D_\sigma$ is a permutation matrix. The first condition ensures that $R$ cannot be made diagonal through row permutations, while the second ensures that each column of $R$ has strictly positive total weight.
We now show that $\widehat R := R(\hat{\mathbf{z}}_n,\mathbf{Z}_n) \notin \mathcal{R}_\epsilon$. Suppose, by contradiction, that $\widehat R \in \mathcal{R}_\epsilon$. Let $\widehat X := X(\hat{\mathbf{z}}_n,\mathbf{Z}_n)$ and $X := X(\mathbf{Z}_n,\mathbf{Z}_n)$. Then, by the definition of $\hat{\mathbf{z}}_n$ as a maximizer of $Q$ over the set $ \mathcal{F}(n, \alpha)$ and since $\alpha < \min_a \pi_a$, we have
\begin{equation} \label{eq:Q_diff}
Q(\hat{\mathbf{z}}_n) - Q(\mathbf{Z}_n) \ge 0 \quad \text{eventually almost surely.}
\end{equation}
Define
\[
\eta_n := \inf_{R \in \mathcal{R}_\epsilon} \big\{ H_{\rho_n S, n}(\tDiag(R^\intercal \veco)) - H_{\rho_n S, n}(R) \big\}.
\]
By Lemma~\ref{lemma7.2}, we can rewrite the definition of $\eta_n$ in terms of the functional $G_S$ as follows
$$\eta_n = \rho_n \inf_{R \in \mathcal{R}_\epsilon} \big\{ G_{S}(\tDiag(R^\intercal \veco)) - G_{S}(R) \big\}.$$
Let us define the mapping $\Phi(R) := G_{S}(\tDiag(R^\intercal \veco)) - G_{S}(R)$. We want to prove that $\eta_n \ge c \rho_n$ for some constant $c > 0$, so it suffices to show that $\inf_{R \in \mathcal{R}_\epsilon} \Phi(R) > 0$. We establish this strictly inequality through a compactness and continuity argument: first, the mapping $R \mapsto \Phi(R)$ is continuous on $\mathcal{R}_\epsilon$. Indeed, since $S$ is a fixed positive full rank matrix and $[R^\intercal \veco]_a > 0$ for all $a \in [k]$, the matrix entries are continuous operations and the arguments inside the logarithms of $G_S(R)$ are strictly positive and bounded away from zero. Second, the set $\mathcal{R}_\epsilon$ is a closed and bounded subset of the finite-dimensional space of matrices (constrained by the norm $\|R\|_1 = 1$), which ensures that $\mathcal{R}_\epsilon$ is compact. By the extreme value theorem, the continuous function $\Phi$ attains its minimum on the compact set $\mathcal{R}_\epsilon$. Therefore, there exists a matrix $R^* \in \mathcal{R}_\epsilon$ such that$$c := \inf_{R \in \mathcal{R}_\epsilon} \Phi(R) = \Phi(R^*).$$Finally, by the definition of $\mathcal{R}_\epsilon$, any matrix within this set is strictly bounded away from the set of diagonal matrices up to permutation. Consequently, $R^*$ cannot be a diagonal matrix up to permutation. According to the strict inequality condition established in Lemma~\ref{lemma7.2}, $\Phi(R) = 0$ if and only if $R$ is diagonal up to permutation. Since $R^* \in \mathcal{R}_\epsilon$, we necessarily have $c = \Phi(R^*) > 0$. Multiplying by $\rho_n$ yields $\eta_n \ge c \rho_n$. Hence,
\begin{equation} \label{eq:H_diff}
H_{\rho_n S, n}(\tDiag(\widehat R^\intercal \veco)) - H_{\rho_n S, n}(\widehat R) \ge c \rho_n.
\end{equation}
Combining \eqref{eq:Q_diff} and \eqref{eq:H_diff}, by the definition \eqref{defX} of $X$, we obtain
\[
\widehat X - X \ge c \rho_n.
\]
However, Theorem~\ref{theorem7.6} ensures that for any $\delta>0$,
\[
|\widehat X - X| < \delta \rho_n,
\]
eventually almost surely, which is a contradiction if $\delta < c$. Therefore, $\widehat R \notin \mathcal{R}_\epsilon$ eventually almost surely. 
The second condition of $\mathcal{R}_\epsilon$ is automatically satisfied because $\hat{\mathbf{z}}_n \in \mathcal{F}(n, \alpha)$. Consequently, the first condition cannot hold, implying
\[
\big\| \widehat R - \tDiag(\widehat R^\intercal \veco) \big\|_1 < \epsilon.
\]
Thus, we conclude that
\[
\mathbb{P}\Big( \| R(\hat{\mathbf{z}}_n, \mathbf{Z}_n) - R(\mathbf{Z}_n, \mathbf{Z}_n) \|_1 < \epsilon \Big) \longrightarrow 1,
\]
from which weak consistency follows
\[
\mathbb{P}\big( m(\hat{\mathbf{z}}_n,\mathbf{z}_n) < \epsilon n \big) \longrightarrow 1.
\]

\section{Strong consistency results}\label{sect_aux_strong}

This section gathers the technical results required for the proof of Theorem \ref{thm_strong} and the proof of this theorem. We first establish concentration and deviation bounds for the edge count statistics, and then derive a control of the likelihood difference. Finally, we analyze the local behavior of the population criterion through a variational argument.

\subsection{Concentration and deviation bounds}

To establish these results, we define, for any $a,b \in [k]$ and any labelings $\mathbf{e}_{n},\mathbf{z}_{n} \in [k]^n$, the quantity
\begin{equation}\label{def_W}
W_{ab}(\mathbf{e}_n,\mathbf{z}_n) =\frac{1}{n^2} \left\{ o_{ab}(\mathbf{e}_n) - \mathbb{E}[o_{ab}(\mathbf{e}_n)\mid \mathbf{z}_n] - o_{ab}(\mathbf{z}_n) + \mathbb{E}[o_{ab}(\mathbf{z}_n)\mid \mathbf{z}_n] \right\} .
\end{equation}

We begin by establishing concentration inequalities for the deviations of empirical edge counts from their expectations. These results will be used to control the fluctuations uniformly over all possible labelings.

\begin{prop}
\label{prop7.9}
Let $(\mathbf{Z}_n,\mathbf{A}_{n \times n}) $ be generated from a DC-SBM with parameters $(k,\pi, \rho_nS, \boldsymbol{\theta})$ with $\rho_n=1$ or $\rho_n\to0$ as $n\to\infty$. For $\mathbf{z}_n, \mathbf{e}_n \in [k]^n$  and for 
\begin{align*} 
W(\mathbf{e}_n,\mathbf{z}_n)  = &\sum_{1\leq a, b\leq k}\, W_{ab}(\mathbf{e}_n,\mathbf{z}_n) \log \Bigl( \rho_n[R(\mathbf{z}_{n},\mathbf{z}_{n})SR(\mathbf{z}_{n},\mathbf{z}_{n})]_{ab }\Bigr) \\
&-  2  \sum_{1\leq a,b \leq k}\,  W_{ab}(\mathbf{e}_n,\mathbf{z}_n)    \log  \Bigl( \rho_n[R(\mathbf{z}_{n},\mathbf{z}_{n})SR(\mathbf{z}_{n},\mathbf{z}_{n})\mathbf 1]_{a }\Bigr),  \\
\end{align*}
we have that 
\begin{align*}
\mathbb{P}&\Bigl(W(\mathbf{e}_n,\mathbf{z}_n) >x\,|\, \mathbf{z}_n\Bigr)\\
& \;\leq\; \exp\Bigl[ - \sup_{t\in (0,1 ]}  \Bigl\{n^2 xt-  \rho_n n^2 C_t[R(\mathbf{e}_{n},\mathbf{z}_{n}),S]\Bigr\}  +\tilde D \rho_n n m(\mathbf{e}_{n},\mathbf{z}_{n}) +D\rho_nm(\mathbf{e}_{n},\mathbf{z}_{n})^2 \Bigr] ,
\end{align*}
with  
\begin{equation}\label{tildeC} 
C_t(R,S) = \sum_{a}\sum_{b\neq b'} [R^\intercal \mathbf 1]_a  K_t(S_{ab} \| S_{ab'})R_{bb'}    ,
\end{equation}
$K_t(p \| q) = p^t q^{1-t} - t q \log(p/q) - q$, while \( D \) and \( \tilde{D} \) are two constants depending only on \( S \).
\end{prop}

\begin{proof}
The proof can be found in Section \ref{proof_prop7.9}.
\end{proof}

The next result provides a finer control at the level of individual block counts.

\begin{prop}
\label{propB.2}
Let $(\mathbf{Z}_n,\mathbf{A}_{n \times n}) $ be generated from a DC-SBM with parameters $(k,\pi, \rho_nS, \boldsymbol{\theta})$. For $W_{ab}(\mathbf{e}_n,\mathbf{z}_n)$ defined as in \eqref{def_W} and for all $x>0$, we have 
\begin{equation*}
 \mathbb{P}\Bigl( |W_{ab}(\mathbf{e}_n,\mathbf{z}_n) | > x \mid\mathbf{z}_n \Bigr) \;\leq\;2\exp\bigl(- xn^2 + 12\rho_ns_{\max} m(\mathbf{e}_n,\mathbf{z}_n)n\bigr).
 \end{equation*}
\end{prop}

\begin{proof}
The proof can be found in Section \ref{proof_propB.2}.
\end{proof}

We now strengthen the previous bound to obtain a uniform control over all community assignments.

\begin{prop}
\label{prop7.7}
Let $(\mathbf{Z}_n,\mathbf{A}_{n \times n}) $ be generated from a DC-SBM with parameters $(k,\pi, \rho_nS, \boldsymbol{\theta})$. For $W_{ab}(\mathbf{e}_n,\mathbf{z}_n)$ defined as in \eqref{def_W}, and for $\rho_n\geq \log n/n$, there exists a constant $c>0$, only depending on $S$, such that, given $\mathbf{Z}_n=\mathbf{z}_n$
$$
 |W_{ab}(\mathbf{e}_n,\mathbf{z}_n) | \leq  \frac{c \rho_n m(\mathbf{e}_n,\mathbf{z}_n)}{n},
$$
simultaneously for all $a,b \in [k]$ and $\mathbf{e}_n \in[k]^n$, eventually almost surely as $n \to \infty$.
Moreover, there exists a constant $c' > 0$, depending only on $S$, such that
\begin{equation*}
\left| \frac{o_{ab}(\mathbf{e}_n)}{n^2} - \frac{o_{ab}(\mathbf{z}_n)}{n^2} \right| \leq \frac{c' \rho_n m(\mathbf{e}_n, \mathbf{z}_n)}{n},
\end{equation*}
simultaneously for all $a, b \in [k]$ and $\mathbf{e}_n \in [k]^n$, almost surely for all sufficiently large $n$.
\end{prop}

\begin{proof}
The proof can be found in Section \ref{proof_prop7.7}.
\end{proof}

\subsection{Control of the likelihood difference}

We now combine the previous concentration results to control the difference of the likelihood criterion evaluated at an arbitrary labeling and at the true labeling.
To establish this, we introduce the event $\mathcal{C}$
\begin{align}\label{def_C}
\mathcal{C} & = \left\{ \sup_{\mathbf{e}_n \in [k]^n} \sup_{a,b \in [k]} \left| \frac{o_{ab}(\mathbf{e}_n)}{n^2} - \frac{\mathbb{E} [o_{ab}(\mathbf{e}_n)\mid \mathbf{Z}_n]}{n^2} \right| \leq \sqrt{\frac{8 s_{\max} \rho_n \log k}{n^2}} \right\}\\
& = \left\{ \sup_{\mathbf{e}_n \in [k]^n} \sup_{a,b \in [k]} \left| \frac{o_{ab}(\mathbf{e}_n)}{n^2} -\rho_n [R(\mathbf{e}_n,\mathbf{Z}_n)SR(\mathbf{e}_n,\mathbf{Z}_n)] \right| \leq \sqrt{\frac{8 s_{\max} \rho_n \log k}{n^2}} \right\}, \nonumber
\end{align}
where $s_{\max}$ denotes the maximum entry of the matrix $S$.

\begin{theorem}
\label{thm7.10}
Let $(\mathbf{Z}_n,\mathbf{A}_{n \times n}) $ be generated from a DC-SBM with parameters $(k,\pi, \rho_nS, \boldsymbol{\theta})$ with $\rho_n \geq  \log n/n$. Let $\mathbf{z}_n\in [k]^n$ and $\mathbf{e}_n \in \mathcal F(n,\alpha)$ with $\alpha<\min_{a} \pi_a$. For all $ \epsilon >0$ sufficiently small, if $m=m(\mathbf{e}_n,\mathbf{z}_n)\leq  \epsilon n $ and $n$ sufficiently large, then
\begin{align*}
\mathbb P \left(  \left\{ X(\mathbf{e}_n,\mathbf{z}_n)  - X(\mathbf{z}_n,\mathbf{z}_n) > x \right\} \cap \mathcal C \mid  \mathbf z_n\right) \;\leq\;  \exp\Bigl[  - \sup_{t\in(0,1 ]} \Bigl\{ n^2(xt-  \rho_n  C_t(R,S))\Bigr\} +\phi(n,m) \Bigr]\,,
\end{align*} 
with $\mathcal C$ defined in \eqref{def_C}, $C_t(R,S)$ given by \eqref{tildeC}, and with $\phi(n,m) = C(\rho_n m^2 +\rho_n n m +1)$
and $C$ a constant depending only on $S$ and $k$.
\end{theorem}

\begin{proof}
The proof can be found in Section \ref{proof_thm7.10}.
\end{proof}

\subsection{Variational analysis of the criterion}

This part takes its roots in the work of  \citep{vander}, later adapted by \citep{CL24}.
We now study the local behavior of the population criterion by considering small perturbations of the matrix $R$. The following lemmas provide a first-order expansion and a corresponding lower bound. We first compute the directional derivative of the population criterion.

\begin{lemma} \label{lemma_der}
Let
$$G(\lambda)\;=\;  H_{S,n}\Bigl(  \tDiag(R^\intercal\veco) + \sum_{b\neq b'}\lambda_{bb'}\Delta_{bb'}    \Bigr)\,$$
for some given $k \times k$  matrices $\Delta_{bb'}$ defined by, for $1 \leq b \neq b'\leq k$
$$(\Delta_{bb'})_{b'b'} = -1, \qquad
(\Delta_{bb'})_{bb'} = 1,$$
and
$$(\Delta_{bb'})_{aa'} = 0 \quad \text{for all other entries } (a,a') \in [k]^2.$$
Then we have that 
$$\frac{\partial}{\partial \lambda_{bb'}} G(0)\;=\; - 2\rho_n \left[ \sum_{a} r_{a} KL(S_{ab'}\| S_{ab}) - KL((Sr)_{b'}\|(Sr)_{b})\right],$$
where $KL$ is defined in \eqref{def_KL_poiss} and with $r=[R^\intercal\veco]$.
\end{lemma}

\begin{proof}
The proof can be found in Section \ref{proof_lemma_der}.
\end{proof}

The previous result allows us to control the variation of the criterion and derive the following lower bound.

\begin{lemma}
 \label{lemma7.3} 
Let $S$ be a fixed symmetric matrix with strictly positive entries, and let $\rho_n=1$ or $\rho_n \rightarrow 0$ as $n\rightarrow \infty$. For all $\epsilon >0$ and for all $R$ with $\|R-\tDiag(\pi)\|_1 < \delta$ for $\delta$ sufficiently small, then   
\begin{align*}
H_{S,n}(\tDiag(R^\intercal \veco)) -  &H_{S,n}(R)\\
& \geq 2 \rho_n \sum_{b\neq b'}\left( \sum_{a} [R^\intercal \veco]_{a} KL(S_{ab'}\| S_{ab}) -  KL([SR^\intercal \veco]_{b'}\|[SR^\intercal \veco]_{b}) - \epsilon\,\right)R_{b'b},
\end{align*}
with $KL$ defined in \eqref{def_KL_poiss}. 
\end{lemma}

\begin{proof}
The proof follows the same argument as Lemma C.4 in \citep{CL24}, substituting the derivative of the function $G$ defined in Lemma \ref{lemma_der} with the specific expression established therein.
\end{proof}

\subsection{Proof of Theorem \ref{thm_strong} (strong consistency)} \label{sec_proof_strong}

Let $m=m(\hat{\mathbf{z}}_n,\mathbf{Z}_{n} )$ and $\hat R = R(\hat{\mathbf{z}}_n,\mathbf{Z}_{n})$. Take $\delta>0$ such that 
$$m\;=\;  \frac{n}{2}\,\|\tDiag(\hat R^\intercal \mathbf 1) - \hat R\|_1  \;\leq \; \delta n\,.$$
Observe that this event holds with increasing probability by Theorem~\ref{thm_weak}, so it is enough to prove that $m$ does not belong to the interval $(0, \delta n]$, with probability converging to 1. 
By the strong law of large numbers, we also have that 
$$\|\tDiag(\hat R^\intercal \veco) - \tDiag(\pi)\|_1 \;\leq \; \delta ,$$
eventually almost surely as $n\to\infty$. Therefore the above facts imply that 
$$\|\hat R  - \tDiag(\pi)\|_1 \;\leq \; 2\delta,$$
with probability converging to 1 as $n\to\infty$. For any $\epsilon >0$, if $\delta$ is taken sufficiently small, by Lemma~\ref{lemma7.3}, we have that 
\begin{equation}\label{mainineq}
\begin{split}
H_{S,n}(\tDiag(\hat R^\intercal \veco)) -  H_{S,n}(\hat R) \;&\geq \; \rho_n \widetilde C(\hat R,S) ,
\end{split}
\end{equation}
with 
\[
\widetilde C(R, S) = 2 \sum_{b\neq b'} \left( \sum_{a}[R^\intercal \veco]_{a} KL(S_{ab'}\| S_{ab}) -  KL([SR^\intercal \veco]_{b'}\|[SR^\intercal \veco]_{b}) - \epsilon\ \right)R_{bb'} \,,
\]
for $KL$ defined as in \eqref{def_KL_poiss}. 
Moreover, since $\alpha< \min_{a}\pi_a$ we have that
\begin{equation}\label{eq378}
Q(\hat{\mathbf{z}}_n)  - Q(\mathbf{Z}_n)   \;\geq\; 0,
\end{equation}
eventually almost surely as $n\to\infty$. Then by summing 
\eqref{mainineq} and \eqref{eq378} 
 we must have, by \eqref{defX}, that 
\begin{equation*}
 X(\hat{\mathbf{z}}_n,\mathbf{Z}_n)  - X(\mathbf{Z}_n,\mathbf{Z}_n) \;\geq \;  \rho_n  \widetilde C (\hat R,S) ,
\end{equation*}
eventually almost surely as $n\to\infty$.
We now define, for an arbitrary small $\eta>0$, the event
$$
\mathcal B(\eta) = \left\{\sup_{a\in [k]} \left| \frac{n_a(\mathbf{Z}_n)}{n}-\pi_a\right|< \eta \right\}.$$
We know by the strong law of large numbers, that $\mathcal B(\eta) $ holds, for any $\eta >0$, eventually almost surely as $n \to \infty$. Thus in the remaining of the proof, we will assume that $\mathbf{Z}_n \in \mathcal B(\eta) $. On this set we will prove that
$$ \mathbb P \left( \left\{  X(\hat{\mathbf{z}}_n,\mathbf{Z}_n)  - X(\mathbf{Z}_n,\mathbf{Z}_n) \geq \rho_n \widetilde C(\hat R,S) \right\} \cap \left\{ m(\hat{\mathbf{z}}_n,\mathbf{Z}_n) \in (0,\delta n]\right\} \right) \to 0,$$ 
as $n \to \infty$, implying that $m(\hat{\mathbf{z}}_n,\mathbf{Z}_n)=0$ with probability converging to 1 as $n\to\infty$. 
Let $\mathbf{Z}_n \in \mathcal B(\eta) $ and let $\mathbf{e}_n \in \mathcal F(n,\alpha)$ with $m(\mathbf{e}_n,\mathbf{Z}_n) \leq \delta n$. Conditioning on the event $\mathbf{Z}_n=\mathbf{z}_n$ and applying a union bound over the possible values of $m$ and $\mathbf{e}_n$, by noting $R = R(\mathbf{e}_n,\mathbf{Z}_{n})$, we obtain
\begin{align}\label{eq_P}
& \mathbb P \left( \left\{  X(\hat{\mathbf{z}}_n,\mathbf{Z}_n)  - X(\mathbf{Z}_n,\mathbf{Z}_n) \geq \rho_n \widetilde C(\hat R,S) \right\} \cap \left\{ m(\hat{\mathbf{z}}_n,\mathbf{Z}_n) \in (0,\delta n]\right\} \right) \nonumber\\\nonumber
&\leq \sum_{\mathbf{z}_n} \mathbb P \left(  \sup_{m\in[1, \delta n]} \sup_{\mathbf{e}_n : m(\mathbf{e}_n,\mathbf{z}_n)=m   } X(\mathbf{e}_n,\mathbf{z}_n)  - X(\mathbf{z}_n,\mathbf{z}_n) > \rho_n \widetilde C(R,S)  \mid \mathbf Z_n= \mathbf z_n\right) \mathbb P (\mathbf Z_n= \mathbf z_n) \\ 
&\leq \sum_{\mathbf{z}_n} \sum_{m \in [1,\delta n ]}\sum_{\mathbf{e}_n :m(\mathbf{e}_n,\mathbf{z}_n)=m } \mathbb P \left(  \left\{ X(\mathbf{e}_n,\mathbf{z}_n)  - X(\mathbf{z}_n,\mathbf{z}_n) > \rho_n \widetilde C(R,S) \right\} \cap \mathcal C \mid \mathbf Z_n= \mathbf z_n\right) \mathbb P (\mathbf Z_n= \mathbf z_n) \nonumber \\
& \hspace{5cm}+ \mathbb P (\mathcal C ^c),
\end{align}
with $\mathcal C$ defined in \eqref{def_C}. The event $\mathcal C$ holds with probability converging to one as $n \to \infty$ by Theorem \ref{thm_7.4}.
Then, if $\delta$ is sufficiently small, for any $\mathbf e_n$ with $m=m(\mathbf{e}_n,\mathbf{Z}_n) \leq \delta n$, by Theorem~\ref{thm7.10} we have that
\begin{align*}
&\mathbb P \left(  \left\{ X(\mathbf{e}_n,\mathbf{z}_n)  - X(\mathbf{z}_n,\mathbf{z}_n) > \rho_n \widetilde C(R,S) \right\} \cap \mathcal C \mid \mathbf Z_n= \mathbf z_n\right) \\
&\;\leq\; \exp\Bigl[ - \sup_{t\in(0,1 ]} \Bigl\{  n^2 \rho_n (t \widetilde C(R,S) -  C_t(R,S))\Bigr\} + \phi(n,m) \Bigr],
\end{align*}
with $C_t(R,S)$  given by \eqref{tildeC}, and $\phi(n,m)=C(\rho_nm^2+\rho_nm+1)$. Remark that $\phi(n,m)\leq  \epsilon' \rho_n m n$, for any $\epsilon' > 0$ and for $n$ large enough. So we obtain
\begin{align*}
&\mathbb P \left(  \left\{ X(\mathbf{e}_n,\mathbf{z}_n)  - X(\mathbf{z}_n,\mathbf{z}_n) > \rho_n \widetilde C(R,S) \right\} \cap \mathcal C \mid \mathbf Z_n= \mathbf z_n\right) \\
&\;\leq\; \exp\Bigl[ - \sup_{t\in(0,1 ]} \Bigl\{  n^2 \rho_n (t \widetilde C(R,S) -  C_t(R,S))\Bigr\} + \epsilon' \rho_n m n  \Bigr].
\end{align*}
Recall $K_t(p \| q) = p^t q^{1-t} - t q \log(p/q) - q$. Observe that for $t\in(0,1 ]$, we can write
\begin{align*}
t \widetilde C(R,S) -  C_t(R,S)=&2t \sum_{b\neq b'} \left( \sum_{a}[R^\intercal \veco]_{a} KL(S_{ab'}\| S_{ab}) -  KL([SR^\intercal \veco]_{b'}\|[SR^\intercal \veco]_{b}) - \epsilon\ \right)R_{bb'}  \\
& \qquad -\left[ \sum_{a}\sum_{b\neq b'} [R^\intercal \mathbf 1]_a  K_t(S_{ab} \| S_{ab'})R_{bb'}     \right]\nonumber \\
=&  \sum_{b\neq b'} \Biggl[  \sum_{a}[R^\intercal \veco]_{a}  \left( tKL(S_{ab'} \| S_{ab}) - K_t(S_{ab} \| S_{ab'}) \right)+ t\sum_{a}[R^\intercal \veco]_{a} KL(S_{ab'}\| S_{ab}) \\
&\qquad- 2t  KL([SR^\intercal \veco]_{b'}\|[SR^\intercal \veco]_{b})  -2t\epsilon \Biggr] R_{bb'}  \nonumber \\
=& \sum_{b\neq b'} \Biggl[
\sum_{a}[R^\intercal \veco]_{a} H_t(S_{ab'} \| S_{ab})
+ t\sum_{a}[R^\intercal \veco]_{a} KL(S_{ab'}\| S_{ab}) \\
&\qquad- 2t KL([SR^\intercal \veco]_{b'}\|[SR^\intercal \veco]_{b})
\Biggr] R_{bb'} - 2t\epsilon \frac m n , \nonumber \\
\end{align*}
with $H_t(p\| q) = tKL(p\|q) -K_t(q\|p)=(1-t) p + t q - p^{1-t}q^{t}\,$, and by using $\sum_{b\neq b'} R_{bb'}=m(\mathbf{e}_n,\mathbf{z}_n)/n$.
By using the log sum inequality and the fact that $[SR^\intercal \veco]_{b}= \sum_ a S_{ab} [R^\intercal \veco]_{a}$, we obtain that
$$  \sum_{a}[R^\intercal \veco]_{a} KL(S_{ab'}\| S_{ab}) \geq  KL([SR^\intercal \veco]_{b'}\|[SR^\intercal \veco]_{b}).$$
Therefore
\begin{align*}
t \widetilde C(R,S) -  C_t(R,S)\nonumber 
  \geq & \sum_{b\neq b'} \Biggl[  \sum_{a}[R^\intercal \veco]_{a}  H_t(S_{ab'} \| S_{ab}) -  t  KL([SR^\intercal \veco]_{b'}\|[SR^\intercal \veco]_{b}) \Biggr] R_{bb'} -2 t\epsilon \frac m n  \nonumber \\
 \geq& \left(\min_{b\neq b'} \left[ \sum_a [R^\intercal \veco]_{a} H_t(S_{ab'} \| S_{ab}) -t   KL([SR^\intercal \veco]_{b'}\|[SR^\intercal \veco]_{b}) \right]  - 2t\epsilon \right)\frac mn \nonumber,
\end{align*}
where we used again $\sum_{b\neq b'} R_{bb'}=m(\mathbf{e}_n,\mathbf{z}_n)/n$.
Since \(\mathbf{Z}_n \in \mathcal{B}(\eta)\), we have \([R^\intercal \mathbf{1}]_a \to \pi_a\), so we obtain
$$ \sup_{t\in(0,1]}  t \widetilde C(R,S) -  C_t(R,S) \geq \left( C(\pi,S)- 2\epsilon \right) \frac mn ,$$
with
\begin{equation*}
C(\pi, S) \;=\;   \min_{b\neq b'}\; \sup_{t\in [0,1]}  \left[ \sum_a \pi_{a} H_t(S_{ab'} \| S_{ab}) -t   KL([S\pi]_{b'}\|[S\pi]_{b}) \right]  \,.
\end{equation*}
By Lemma \ref{lemma_C_positif}, $C(\pi, S) >0$ when $S$ is full rank.
Therefore 
\begin{equation} \label{bound_P} 
\mathbb P \left(  \left\{ X(\mathbf{e}_n,\mathbf{z}_n)  - X(\mathbf{z}_n,\mathbf{z}_n) > \rho_n \widetilde C(R,S) \right\} \cap \mathcal C \mid \mathbf Z_n= \mathbf z_n\right) \leq \exp \left[ -(C(\pi,S)-\epsilon)\rho_n n m \right],
\end{equation}
with $\epsilon$ arbitrarily small. The number of assignments $\mathbf{e}_n \in [k]^n$ such that 
$m(\mathbf{e}_n,\mathbf{z}_n)=m$ is given by
\[
\binom{n}{m}(k-1)^m,
\]
since one must choose the $m$ misclassified node indices among $n$, and assign to each a label different from the true one. Using the standard bound $\binom{n}{m} \le (en/m)^m$, this yields
\[
\binom{n}{m}(k-1)^m \le \left(\frac{en(k-1)}{m}\right)^m.
\]
We use this and the bound \eqref{bound_P} in \eqref{eq_P} to finally obtain that 
\begin{align}\label{fin}
& \mathbb P \left( \left\{  X(\hat{\mathbf{z}}_n,\mathbf{Z}_n)  - X(\mathbf{Z}_n,\mathbf{Z}_n) \geq \rho_n \widetilde C(\hat R,S) \right\} \cap \left\{ m(\hat{\mathbf{z}}_n,\mathbf{Z}_n) \in (0,\delta n]\right\} \right) \nonumber\\\nonumber
&\leq \sum_{m=1}^{\delta n} \left( \frac{en(k-1)}{m} \right)^m \exp \left\{ -[C(\pi,S)-\epsilon]\rho_n n m \right\} + \mathbb P \left( \mathcal C ^c\right) \\
&\leq \sum_{m=1}^{\delta n} \exp \left( -m \{[C(\pi,S)-\epsilon]\rho_n n -\log n -\log [e(k-1)]\} \right)+ \mathbb P \left( \mathcal C ^c\right) .
\end{align}
Now, for $\rho_n \gg \log n/n$ (including $\rho_n=1$) and $C(\pi,S) > 0$ or $\rho_n = \log n/n$ and $C(\pi,S) > 1$, and by Theorem \ref{thm_7.4} that implies $ \mathbb P \left( \mathcal C ^c\right) \to 0$, the last expression in \eqref{fin} converges to $0$ as $n\to \infty$.

\section{Proof of auxiliary results for Theorem \ref{thm_weak} (weak consistency)}\label{proof_aux_weak}

This section is devoted to the proof of the auxiliary results used in establishing Theorem \ref{thm_weak}.

\subsection{Proof of Lemma \ref{lemma7.1} }\label{proof_lemma7.1}

Following Lemma 1 in \citep{vander}
, fix $\mathbf e_n, \mathbf z_n \in [k]^n$. We begin by rewriting the left-hand side 
\begin{align*}
\frac{1}{n}\sum_{i=1}^n \theta_i \mathbf 1_{\{e_i \neq z_i\}}
&= \frac{1}{n}\sum_{i=1}^n \theta_i - \frac{1}{n}\sum_{i=1}^n \theta_i \mathbf 1_{\{e_i = z_i\}} \\
&= 1 - \frac{1}{n}\sum_{i=1}^n \theta_i \mathbf 1_{\{e_i = z_i\}}.
\end{align*}
By definition of $R(\mathbf e_n,\mathbf z_n)$, we have
\[
\frac{1}{n}\sum_{i=1}^n \theta_i \mathbf 1_{\{e_i = z_i\}}
= \sum_{a=1}^k [R(\mathbf e_n,\mathbf z_n)]_{aa},
\]
and therefore
\begin{align}
\frac{1}{n}\sum_{i=1}^n \theta_i \mathbf 1_{\{e_i \neq z_i\}}
&= 1 - \sum_{a=1}^k [R(\mathbf e_n,\mathbf z_n)]_{aa} \nonumber \\
&= \sum_{a=1}^k \left( \frac{n_a(\mathbf z_n)}{n} - [R(\mathbf e_n,\mathbf z_n)]_{aa} \right),
\label{eq:first_identity}
\end{align}
since $\sum_{a} n_a(\mathbf z_n)/n = 1$.
We now turn to the $\ell^1$ norm. Observe that
\begin{align*}
\frac{1}{2}\left\|
\tDiag\bigl(R(\mathbf e_n,\mathbf z_n)^\intercal \mathbf 1\bigr)
- R(\mathbf e_n,\mathbf z_n)
\right\|_1
&= \frac{1}{2} \sum_{a,b=1}^k
\left|
\left[\tDiag\bigl(R(\mathbf e_n,\mathbf z_n)^\intercal \mathbf 1\bigr)\right]_{ab}
- [R(\mathbf e_n,\mathbf z_n)]_{ab}
\right|\\
&= \frac{1}{2}
\left[
\sum_{a=1}^k
\left|
\frac{n_a(\mathbf z_n)}{n} - [R(\mathbf e_n,\mathbf z_n)]_{aa}
\right|
+
\sum_{a \neq b}
[R(\mathbf e_n,\mathbf z_n)]_{ab}
\right].
\end{align*}
Note that for each $a\in [k]$,
\[
\frac{n_a(\mathbf z_n)}{n}
\ge
[R(\mathbf e_n,\mathbf z_n)]_{aa}.
\]
Moreover,
\[
\sum_{a \neq b}
[R(\mathbf e_n,\mathbf z_n)]_{ab}
=
\frac{1}{n}\sum_{i=1}^n \theta_i \mathbf 1_{\{e_i \neq z_i\}}.
\]
Hence,
\begin{align*}
\frac{1}{2}\left\|
\tDiag\bigl(R(\mathbf e_n,\mathbf z_n)^\intercal \mathbf 1\bigr)
- R(\mathbf e_n,\mathbf z_n)
\right\|_1
&= \frac{1}{2}
\left[
\sum_{a=1}^k
\left(
\frac{n_a(\mathbf z_n)}{n} - [R(\mathbf e_n,\mathbf z_n)]_{aa}
\right)
+
\frac{1}{n}\sum_{i=1}^n \theta_i \mathbf 1_{\{e_i \neq z_i\}}
\right].
\end{align*}
Using \eqref{eq:first_identity}, we conclude that
\[
\frac{1}{2}\left\|
\tDiag\bigl(R(\mathbf e_n,\mathbf z_n)^\intercal \mathbf 1\bigr)
- R(\mathbf e_n,\mathbf z_n)
\right\|_1
=
\sum_{a=1}^k
\left(
\frac{n_a(\mathbf z_n)}{n} - [R(\mathbf e_n,\mathbf z_n)]_{aa}
\right)
=
\frac{1}{n}\sum_{i=1}^n \theta_i \mathbf 1_{\{e_i \neq z_i\}}.
\]
Finally, since $R(\mathbf z_n,\mathbf z_n) = \tDiag(R(\mathbf e_n,\mathbf z_n)^\intercal \mathbf 1)$, the last equality follows 
\[
\frac{1}{2}\|R(\mathbf z_n,\mathbf z_n) - R(\mathbf e_n,\mathbf z_n)\|_1
=
\frac{1}{2}\left\|
\tDiag\bigl(R(\mathbf e_n,\mathbf z_n)^\intercal \mathbf 1\bigr)
- R(\mathbf e_n,\mathbf z_n)
\right\|_1.
\]
This completes the proof.

\subsection{Proof of Lemma \ref{lemma7.2} }\label{proof_lemma7.2}

Expanding $ H_{S,n}(\Diag) -  H_{S,n}(R)$, we obtain 
\begin{align*}
&\frac{1}{\rho_n} \biggl(  H_{S,n}(\Diag)  -  H_{S,n}(R)\biggr) \\
&\quad = \sum_{ab} [\Diag S\Diag^\intercal]_{ab}  \log \left(\frac{[\Diag S\Diag^\intercal]_{ab}   }{ [\Diag S\Diag \mathbf1]_a  [\Diag S\Diag \mathbf1]_b}\right) \\ 
&\quad \qquad - \sum_{ab} [R SR^\intercal]_{ab}  \log \left(\frac{[R SR^\intercal]_{ab}   }{ [R SR^\intercal \mathbf1]_a  [R SR^\intercal \mathbf1]_b}\right)  \\
&\quad \qquad - \log (n^2 \rho_n) \sum_{ab}\left(  [\Diag S\Diag^\intercal]_{ab} - [R SR^\intercal]_{ab} \right)  \\
&\quad = \sum_{ab} [\Diag S\Diag^\intercal]_{ab}  \log \left(\frac{[\Diag S\Diag^\intercal]_{ab}   }{ [\Diag S\Diag \mathbf1]_a  [\Diag S\Diag \mathbf1]_b}\right) \\
&\quad \qquad- \sum_{ab} [R SR^\intercal]_{ab}  \log \left(\frac{[R SR^\intercal]_{ab}   }{ [R SR^\intercal \mathbf1]_a  [R SR^\intercal \mathbf1]_b}\right) ,
\end{align*}
where the last equality holds because $\sum_{ab} [RSR^\intercal]_{ab} = \sum_{ab} [\Diag S \Diag]_{ab}$. Let us denote this common sum by $S^\star$. We observe that 
\[
S^\star = \sum_{ab} [R^\intercal \mathbf1]_a S_{ab} [R^\intercal \mathbf1]_b \geq s_{\min}  \sum_{ab} [R^\intercal \mathbf1]_a  [R^\intercal \mathbf1]_b =s_{\min} >0,
\]
where we used $\|R\|_1=1$. 

To prove the non-negativity of this difference, we introduce two probability distributions on $[k]^2$, namely 
\[
P^R_{ab} = \frac{[RSR^\intercal]_{ab}}{S^\star} \quad \text{and} \quad P^D_{ab} = \frac{[\Diag S\Diag]_{ab}}{S^\star}, \quad \forall a,b \in [k].
\]
Their respective first marginal distributions are given by
\[
p_a^R = \frac{[RSR^\intercal \veco]_a}{S^\star} \quad \text{and} \quad p_a^D = \frac{[\Diag S\Diag^\intercal \veco]_a}{S^\star}, \quad \forall a \in [k].
\]
Using the definition \eqref{def_G_S} of $G_S$, we can write
\begin{align*}
\frac{1}{\rho_n} \biggl(  H_{S,n}(\Diag)  -  H_{S,n}(R)\biggr) = G_S(\Diag) -G_S(R).
\end{align*}
We then can express this difference as a difference between mutual informations. Specifically,
\begin{equation}\label{eq_GS_diff}
G_S(\Diag) - G_S(R) = S^\star \left\{ \sum_{ab} P^D_{ab} \log \left( \frac{P^D_{ab}}{p_a^D p_b^D} \right) - \sum_{ab} P^R_{ab} \log \left( \frac{P^R_{ab}}{p_a^R p_b^R} \right) \right\}.
\end{equation}

Let $\tilde R$ be the matrix $R$ normalized by its columns, namely $\tilde R_{ab} = R_{ab}/[R^\intercal \mathbf 1]_b$. By definition, $\tilde R^\intercal$ is a stochastic matrix. Since $R = \tilde R \Diag$, we can relate $P^R$ and $P^D$ via 
\begin{equation*}
P^R = \tilde R P^D \tilde R^\intercal.
\end{equation*}
We now rely on the Data Processing Inequality (DPI) for mutual information \citep[Section 2.8]{cover_thomas}. Let $(U,V)$ be a pair of random variables jointly distributed according to $P^D$. We introduce a new pair of random variables $(X, Y)$ defined as
\begin{align*}
\mathbb P((X,Y)=(a,b) \mid (U,V)=(u,v)) &= \mathbb P(X=a \mid U=u) \mathbb P(Y=b\mid V=v) \\
&= \tilde R_{au} \tilde R_{bv}.
\end{align*}
By marginalizing over $(U,V)$, the joint distribution of $(X,Y)$ is exactly $P^R$. Using the definition of mutual information for two random variables $I(A;B) = \sum_{a,b} P_{A,B}(a,b) \log \left( \frac{P_{A,B}(a,b)}{P_A(a)P_B(b)} \right)$, Equation \eqref{eq_GS_diff} simplifies to 
\[
G_{S}(\Diag) - G_{S}(R) = S^\star \left\{ I(U;V) - I(X;Y) \right\}.
\]
From our construction, the distributions of those random variables form the following Markov chains $V \rightarrow U \rightarrow X$ and $X \rightarrow V \rightarrow Y$. Applying the DPI sequentially yields 
\begin{align} 
I(U;V) &\geq I(X;V), \label{ineq1} \\
I(X;V) &\geq I(X;Y). \label{ineq2}
\end{align}
Combining these inequalities guarantees that $I(U;V) \geq I(X;Y)$, which proves that $$G_{S}(\Diag) - G_{S}(R) \geq 0.$$
Finally, we demonstrate that this inequality is strict unless $R$ is a diagonal matrix, up to permutation. The inequality $I(U;V) \geq I(X;Y)$ is strict if and only if either \eqref{ineq1} or \eqref{ineq2} is strict. Let us analyze \eqref{ineq1}, which becomes an equality if and only if $I(V;U \mid X) = 0$. 

By expanding the conditional mutual information through the Kullback--Leibler divergence \eqref{def_KL_poiss}, 
and using the conditional entropy $\mathcal H$, we obtain 
\begin{align*}
I(V;U &\mid X) \\
=& \mathcal H(V \mid X) - \mathcal H(V \mid U) \\
=& \sum_{x,u} \tilde R_{xu} p_{u}^D \left(\sum_v \text{KL}\left(\frac {P_{uv}^D }{ p_u^D} \bigg\| \frac{\sum_{u'} \tilde R _{xu'} P_{u'v}^D}{ p_x^R }\right) - \frac{\sum_{v,u'} \tilde R _{xu'} P_{u'v}^D}{p_x^R} +  \frac {\sum_v P_{uv}^D }{ p_u^D} \right).
\end{align*}
Since $\sum_v P_{uv}^D = p_u^D$ and $p_x^R = \sum_{v,u'} \tilde R_{xu'} P_{u'v}^D$, the linear terms cancel out, leaving 
\begin{align*}
I(V;U \mid X) = \sum_{x,u,v} \tilde R_{xu} p_{u}^D  \text{KL}\left( \frac {P_{uv}^D }{ p_u^D}\bigg\| \frac{\sum_{u'} \tilde R _{xu'} P_{u'v}^D}{ p_x^R }\right).
\end{align*}
For this sum to be zero, one of two conditions must hold for every $(x,u)\in [k]^2$: either $\tilde R_{xu} p_u^D = 0$, or the two Poisson parameters inside the KL divergence are identical for all $v\in [k]$. 

Assume for contradiction that $\tilde R_{xu}\, p_u^D = 0$ for all $(x,u)$. Since $\bigl[R^\intercal \veco\bigr]_u > 0$ for all $u$, and $p_u^D \geq s_{\min}\bigl[R^\intercal \veco\bigr]_u$, it follows, because $s_{\min} > 0$, that $p_u^D > 0$ for all $u$. Hence, the condition $\tilde R_{xu}\, p_u^D = 0$ for all $(x,u)$ would imply that $R \equiv 0$, in contradiction with the normalization constraint $\|R\|_1 = 1$.

Therefore, $I(V;U \mid X) = 0$ if and only if, for every $(x,u)$ where $\tilde R_{xu} > 0$, the following holds,
\begin{align} \label{eq_strict_cond}
\forall v \in [k], \quad \frac{P_{uv}^D}{p_u^D}= \frac{\sum_{u'} \tilde R_{xu'} P_{u'v}^D}{p_x^R}.
\end{align}
We now show that this implies that $R$ must necessarily be a diagonal matrix up to permutation. Assume a specific row $x$ (a cluster) contains nodes originating from two distinct groups, $u_1$ and $u_2$ (i.e., $\tilde R_{xu_1} > 0$ and $\tilde R_{xu_2} > 0$). Equation \eqref{eq_strict_cond} forces 
\[
\frac{P_{u_1v}^D}{p_{u_1}^D} = \frac{P_{u_2v}^D}{p_{u_2}^D} \quad \forall v,
\]
which simplifies to $S_{u_1v} = c S_{u_2v}$ for a constant $c = [S\Diag \mathbf1]_{u_1}/ [S\Diag \mathbf1]_{u_2}$ independent of $v$. This implies two rows of $S$ are collinear, contradicting the fact that $S$ is a full-rank matrix. 

Thus, $u_1$ must equal $u_2$. This means each row $x$ in $R$ contains at most one non-zero entry. Combined with the requirement that no column of $R$ is entirely zero, $R$ must be a diagonal matrix up to permutation. By symmetry, the same logic applies to \eqref{ineq2}. We conclude that $G_{S}(\Diag) - G_{S}(R) > 0$ strictly, unless $R$ is a diagonal matrix up to permutation.

\subsection{Proof of Theorem \ref{thm_7.4} }\label{proof_thm_7.4}

For any $\mathbf{e}_{n}\in [k]^n$ 
and fixed $a,b\in [k]$ and $x>0$, consider the event 
\[B_n^{ab}(\mathbf{e}_{n}) \; = \;\bigl\{\, o_{ab}(\mathbf e _n)- \rho_n n^2 [RSR^\intercal]_{ab}>\, x \bigr\}\,,
\] 
with $R= R(\mathbf{e}_n, \mathbf{Z}_n)$. 
First we deal with the case $a\neq b$.
Conditioning on $ \mathbf{Z}_n =\mathbf{z}_n$, we observe that 
$$  o_{ab}(\mathbf e _n)- \rho_n n^2 [RSR^\intercal]_{ab} = \sum_{1\leq i,j\leq n } \mathbf 1_{\{e_i=a,e_j=b\}} (A_{ij}- \rho_n \theta_i \theta_j S_{z_iz_j})$$
is a sum of $n_{a}(\mathbf{e}_{n})n_{b}(\mathbf{e}_{n})$ independent, zero-mean random variables. Applying Chernoff's bound, and letting $Y_{ij} = \1_{\{e_i=a, e_j=b\}} A_{ij}$, we have 
\begin{align*}
\mathbb P\left(
B_n^{ab}(\mathbf e_n)
\,\middle|\,
\mathbf Z_n=\mathbf z_n
\right)
&\le
\inf_{t>0}
\exp(-tx)
\,
\mathbb E\left[
\exp\left(
t \sum_{i,j}
\left[
Y_{ij} - \mathbb E\left(Y_{ij} \mid \mathbf Z_n=\mathbf z_n\right)
\right]
\right)
\right]
\\
&\le
\inf_{t>0}
\exp(-tx)
\prod_{i,j}
\frac{
\mathbb E \left[
e^{tY_{ij}}
\,\middle|\,
\mathbf Z_n=\mathbf z_n
\right]
}{
\exp\left(
t\,\mathbb E\left(Y_{ij}\mid \mathbf Z_n=\mathbf z_n\right)
\right)
}.
\end{align*}
Since $Y_{ij}$ is Poisson, we find 
\[
\mathbb{E}[e^{tY_{ij}} \mid \mathbf{Z}_n = \mathbf{z}_n] = \exp \left[ \1_{\{e_i=a, e_j=b\}} \theta_i \theta_j \rho_n S_{z_i z_j} (e^t - 1) \right].
\]
It follows that 
\begin{align*}
\mathbb P\left(
B_n^{ab}(\mathbf e_n)
\,\middle|\,
\mathbf Z_n=\mathbf z_n
\right)
&\le
\inf_{t>0}
\exp\left[
-tx
+
\sum_{i,j}
\theta_i \theta_j \rho_n S_{z_i z_j}
\mathbf 1_{\{e_i=a,\,e_j=b\}}
\left(e^t - 1 - t\right)
\right].
\end{align*}
Using the inequality $e^t-t-1 \leq t^2$ for $t\in [0,1]$, we obtain 
\begin{align*}
\mathbb{P}(  B^{ab}_n(\mathbf{e}_{n},x)\,|\, \mathbf{Z}_n=\mathbf{z}_{n}) &  \leq  \inf_{t\in[0,1]} \exp(-tx + t^2 \rho_n s_{\max} n_a( \mathbf{e}_n) n_b (\mathbf{e}_n) ) \\
&  \leq  \inf_{t\in[0,1]} \exp(-tx + t^2 \rho_n s_{\max} n^2).
\end{align*}
Differentiating with respect to $t$, the minimum is attained at 
\begin{equation} \label{def_t}
t = \frac{x}{2 \rho_n s_{\text{max}} n^2},
\end{equation}
and gives the bound 
$$
 \exp \left( -\frac{x^2 n^2}{4 \rho_n s_{\text{max}}}\right).
$$
The same argument can be applied to bound the probability of the lower tail event $\tilde{B}_n^{ab}(\mathbf{e}_n) = \{o_{ab}(\mathbf{e}_n) - \rho_n n^2 [RSR^\intercal]_{ab} < -x\}$. Let's now concentrate on the case where $a=b$.
To apply Chernoff's bound, we must express $o_{aa}(\mathbf{e}_n)$ as a sum of mutually independent random variables by restricting the summation to indices $1 \leq i \leq j \leq n$
\[
o_{aa}(\mathbf e _n) = \sum_{1\leq i,j\leq n } \mathbf{1}_{\{e_i=a,e_j=a\}} A_{ij} = 2 \sum_{1 \leq i < j \leq n} \mathbf{1}_{\{e_i=a, e_j=a\}} A_{ij} + \sum_{1 \leq i \leq n} \mathbf{1}_{\{e_i=a\}} A_{ii}.
\]
Conditioning on $\mathbf{Z}_n = \mathbf{z}_n$, we define the independent random variables $Y_{ij}$ for all $i \leq j$ by $Y_{ij} = 2 \mathbf{1}_{\{e_i=a, e_j=a\}} A_{ij}$,
for $i < j$, and $Y_{ii} = \mathbf{1}_{\{e_i=a\}} A_{ii}$, for $i = j$.
It follows that
\[
o_{aa}(\mathbf e _n) - \rho_n n^2 [RSR^\intercal]_{aa} = \sum_{1 \leq i \leq j \leq n} \left( Y_{ij} - \mathbb{E}[Y_{ij} \mid \mathbf{Z}_n = \mathbf{z}_n] \right).
\]
Similarly, the minimum is attained at 
\[
t^* = \frac{x}{4 \rho_n s_{\max} n^2}.
\]
This yields to the final upper bound for the intra-community case:
\[
\mathbb{P}\left( B^{aa}_n(\mathbf{e}_{n}) \;\middle|\; \mathbf{Z}_n=\mathbf{z}_{n}\right) \leq \exp \left( -\frac{x^2 n^2}{8 \rho_n s_{\max}} \right).
\]

Consequently, for any labeling $\mathbf{e}_n$, we have 
\[
\mathbb{P} \left( \sup_{a,b \in [k]} \left| \frac{o_{ab}(\mathbf{e}_n)}{n^2} - \rho_n[RSR^\intercal]_{ab} \right| > x \mid \mathbf{Z}_n = \mathbf{z}_n \right) \leq 2 k^2 \exp \left( -\frac{x^2 n^2}{4 \rho_n s_{\text{max}}} \right).
\]
To control this probability simultaneously for all $\mathbf{e}_n \in [k]^n$, we set 
\[
x = \sqrt{\frac{16 s_{\text{max}} \rho_n \log k}{n^2}}.
\]
We then obtain 
\begin{align*}
\mathbb{P} \left( \bigcup_{a,b \in [k]} \bigcup_{\mathbf{e}_n \in [k]^n} \left\{ \left| \frac{o_{ab}(\mathbf{e}_n)}{n^2} - \rho_n [RSR^\intercal]_{ab} \right| > x \right\} \mid \mathbf{Z}_n = \mathbf{z}_n \right) \leq 2 k^2 \exp(-n \log k).
\end{align*}
Then by the formula of total probability, we have
\begin{align} \label{borne}
\mathbb{P}& \left( \bigcup_{a,b \in [k]} \bigcup_{\mathbf{e}_n \in [k]^n} \left\{ \left| \frac{o_{ab}(\mathbf{e}_n)}{n^2} - \rho_n [RSR^\intercal]_{ab} \right| > x \right\}  \right)\nonumber \\
&=\sum_{ \mathbf{z}_n \in [k]^n} \mathbb{P} \left( \bigcup_{a,b \in [k]} \bigcup_{\mathbf{e}_n \in [k]^n} \left\{ \left| \frac{o_{ab}(\mathbf{e}_n)}{n^2} - \rho_n [RSR^\intercal]_{ab} \right| > x \right\} \mid \mathbf{Z}_n = \mathbf{z}_n \right) \mathbb{P} \left( \mathbf{Z}_n = \mathbf{z}_n \right) \nonumber \\
& \leq \sum_{ \mathbf{z}_n \in [k]^n} 2 k^2 \exp(-n \log k) \mathbb{P} \left( \mathbf{Z}_n = \mathbf{z}_n \right) \nonumber \\
& \leq 2 k^2 \exp(-n \log k).
\end{align}
As the bound in \eqref{borne} is summable in $n$, the Borel-Cantelli Lemma ensures that
\[
\left| \frac{o_{ab}(\mathbf{e}_n)}{n^2} - \rho_n [RSR^\intercal]_{ab} \right| < \sqrt{\frac{16 s_{\text{max}} \rho_n \log k}{n^2}}
\]
holds simultaneously for all $a, b \in [k]$ and $\mathbf{e}_n \in [k]^n$ with probability one for $n$ large enough. This proves the first inequality. For the second, we use the relation $D_a(\mathbf{e}_n) = \sum_b o_{ab}(\mathbf{e}_n)$, which leads to 
\begin{align*}
\left| \frac{D_a(\mathbf{e}_n)}{ n^2} -\rho_n [RSR^\intercal \veco]_a \right| & = \left| \sum_b \left( \frac{o_{ab}(\mathbf{e}_n)}{n^2} -\rho_n [RSR^\intercal]_{ab} \right) \right| \\
& \leq \sum_b \left| \frac{o_{ab}(\mathbf{e}_n)}{n^2} - \rho_n[RSR^\intercal]_{ab} \right| \\
& \leq k \sqrt{\frac{16 s_{\text{max}} \rho_n \log k}{n^2}}.
\end{align*}
This completes the proof of Theorem \ref{thm_7.4}.

\subsection{Proof of Corollary \ref{cor_7.5} }\label{proof_cor_7.5}

For any $\delta>0$, by Theorem \ref{thm_7.4}, we have that for all $\mathbf{e}_n \in [k]^n$, and with $R = R(\mathbf{e}_n, \mathbf{Z}_n)$,
\begin{equation}\label{ev}
\left| \frac{o_{ab}(\mathbf{e}_n)}{n^2} - \rho_n [RSR^\intercal]_{ab} \right| < \delta \rho_n,
\end{equation}
simultaneously for all $a, b \in [k]$, eventually almost surely as $n \to \infty$ as soon as either $\rho_n=1$ or $\rho_n \to 0$ with $n\rho_n \to \infty$. 
By definition, we have 
\begin{equation*}
\begin{split}
\rho_n n^2 [RSR^\intercal]_{ab} &= \mathbb{E}[o_{ab}(\mathbf{e}_n) \mid \mathbf{Z}_n = \mathbf{z}_n] \\
&= \sum_{a',b'} \rho_n S_{a'b'} \sum_{1 \leq i,j \leq n} \theta_i \theta_j \1_{\{z_i=a', z_j=b'\}} \1_{\{e_i=a, e_j=b\}}.
\end{split}
\end{equation*}
By hypothesis, since $S_{a'b'} \geq s_{\min}$ and $\mathbf{e}_n \in \mathcal{F}(n, \alpha)$, it follows that 
\begin{equation*}
\begin{split}
\rho_n n^2 [RSR^\intercal]_{ab} &\geq \rho_n s_{\min} \left( \sum_{i=1}^n \theta_i \1_{\{e_i=a\}} \right) \left( \sum_{j=1}^n \theta_j \1_{\{e_j=b\}} \right) \\
&\geq \rho_n s_{\min} n_a(\mathbf{e}_n) n_b(\mathbf{e}_n) \\
&\geq \rho_n s_{\min} \alpha^2 n^2.
\end{split}
\end{equation*}
Analogously, using $n_a(\mathbf{e}_n)\leq n$  for all $a$ and $S_{a'b'} \leq s_{\max}$, we obtain 
\begin{equation*}
\rho_n n^2 [RSR^\intercal]_{ab} \leq \rho_n s_{\max} n^2.
\end{equation*}
Relying on \eqref{ev}, we get
\begin{equation*}
\alpha^2 s_{\min} - \delta < [RSR^\intercal]_{ab} - \delta < \frac{o_{ab}(\mathbf{e}_n)}{\rho_n n^2} < [RSR^\intercal]_{ab} + \delta \leq s_{\max} + \delta.
\end{equation*}
Consequently,
\begin{equation*}
\frac{o_{ab}(\mathbf{e}_n)}{\rho_n n^2} \in \left( \alpha^2 s_{\min} - \delta, s_{\max} + \delta \right)
\end{equation*}
simultaneously for all $a, b \in [k]$ and $\mathbf{e}_n \in \mathcal{F}(n, \alpha)$, eventually almost surely as $n \to \infty$. 
For the second inequality, we again note that $D_a(\mathbf{e}_n) = \sum_b o_{ab}(\mathbf{e}_n)$, which leads to 
\begin{equation*}
\sum_b (\alpha^2 s_{\min} - \delta) < \frac{D_a(\mathbf{e}_n)}{\rho_n n^2} < \sum_b (s_{\max} + \delta),
\end{equation*}
and thus
\begin{equation*}
k(\alpha^2 s_{\min} - \delta) < \frac{D_a(\mathbf{e}_n)}{\rho_n n^2} < k(s_{\max} + \delta).
\end{equation*}

\subsection{Proof of Theorem \ref{theorem7.6} }\label{proof_theorem7.6}

Using the relation $\sum_b o_{ab}(\mathbf{e}_n) = D_a(\mathbf{e}_n)$, the term $Q(\mathbf{e}_n)$ can be expressed as 
\[
Q(\mathbf{e}_n) = \sum_{a,b} \frac{o_{ab}(\mathbf{e}_n)}{n^2} \log(o_{ab}(\mathbf{e}_n)) - 2\sum_{a} \frac{D_a(\mathbf{e}_n)}{n^2} \log(D_a(\mathbf{e}_n)).
\]
Let $f(x) = x \log x$ (with the convention $0 \log 0 = 0$). By normalizing by $\rho_n n^2$ inside the logarithms, we can rewrite $Q$ as 
\[
Q(\mathbf{e}_n) = \rho_n \left[ \sum_{a,b} f\left(\frac{o_{ab}(\mathbf{e}_n)}{\rho_n n^2}\right) - 2 \sum_{a} f\left(\frac{D_a(\mathbf{e}_n)}{\rho_n n^2}\right) \right] + C_n,
\]
where $C_n = -\log(\rho_n n^2)\sum_{a,b} o_{ab}(\mathbf{e}_n)/n^2$. Since $\sum_{a,b} o_{ab}(\mathbf{e}_n)$ is the total number of edges in the graph, $C_n$ is independent of the labeling $\mathbf{e}_n$. Similarly, for $H_{ S, n}(R(\mathbf{e}_n, \mathbf{Z}_n))$, we have 
\[
H_{S, n}(R(\mathbf{e}_n, \mathbf{Z}_n)) = \rho_n \left[ \sum_{a,b} f\!\left(\frac{\mathbb{E}[o_{ab}(\mathbf{e}_n) \mid \mathbf{Z}_n]}{\rho_n n^2}\right) - 2 \sum_{a} f\!\left(\frac{\mathbb{E}[D_a(\mathbf{e}_n) \mid \mathbf{Z}_n]}{\rho_n n^2}\right) \right] + C_n.
\]
By definition of $X(\mathbf{e}_n, \mathbf{Z}_n)$ \eqref{defX}, we obtain 
\begin{align*}
X(\mathbf{e}_n, \mathbf{Z}_n) &= \rho_n \sum_{a,b} \left[ f\left(\frac{o_{ab}(\mathbf{e}_n)}{\rho_n n^2}\right) - f\left(\frac{\mathbb{E}[o_{ab}(\mathbf{e}_n) \mid \mathbf{Z}_n]}{\rho_n n^2}\right) \right] \\
&\quad - 2 \rho_n \sum_{a} \left[ f\left(\frac{D_a(\mathbf{e}_n)}{\rho_n n^2}\right) - f\left(\frac{\mathbb{E}[D_a(\mathbf{e}_n) \mid \mathbf{Z}_n]}{\rho_n n^2}\right) \right].
\end{align*}
The function $f$ has bounded derivatives on any interval $[m^-, m^+]$ where $m^- > 0$. From Corollary \ref{cor_7.5}, for any $\delta' > 0$ and for $n$ sufficiently large, the following holds eventually almost surely 
\[
\frac{o_{ab}(\mathbf{e}_n)}{\rho_n n^2} \in [\alpha^2 s_{\min} - \delta', s_{\max} + \delta'] \quad \text{and} \quad \frac{D_a(\mathbf{e}_n)}{\rho_n n^2} \in [k(\alpha^2 s_{\min} - \delta'), k (s_{\max} + \delta')].
\]
By choosing $\delta' = \alpha^2 s_{\min} / 2$ (and assuming that $\alpha$ is small enough so that $\delta' \leq s_{\max}/2$), we ensure that these ratios lie in the compact intervals 
\[
I_1 = \left[\frac{\alpha^2 s_{\min}}{2},\, \frac{3 s_{\max}}{2}\right]
\quad \text{and} \quad
I_2 = \left[\frac{k \alpha^2 s_{\min}}{2},\, \frac{3k s_{\max}}{2}\right],
\]
whose lower bounds are strictly positive by assumption.
Let $M$ be the maximum absolute value of the derivative $f'$ over $I_1 \cup I_2$. Applying the mean value theorem and using the shorthand $R(\mathbf{e}_n) = R(\mathbf{e}_n, \mathbf{Z}_n)$, as well as \eqref{E_oab} and \eqref{E_Da}, we obtain 
\begin{align*}
|X(\mathbf{e}_n, \mathbf{Z}_n) - X(\mathbf{Z}_n, \mathbf{Z}_n)| &\leq M \Biggl( \sum_{a,b} \left| \frac{o_{ab}(\mathbf{e}_n)}{n^2} - \rho_n[R(\mathbf{e}_n) S R(\mathbf{e}_n)^\intercal]_{ab} \right| \\
&\quad + \sum_{a,b} \left| \frac{o_{ab}(\mathbf{Z}_n)}{n^2} - \rho_n[R(\mathbf{Z}_n) S R(\mathbf{Z}_n)^\intercal]_{ab} \right| \\
&\quad + 2 \sum_{a} \left| \frac{D_a(\mathbf{e}_n)}{n^2} -\rho_n [R(\mathbf{e}_n) S R(\mathbf{e}_n)^\intercal \mathbf 1]_a \right| \\
&\quad + 2 \sum_{a} \left| \frac{D_a(\mathbf{Z}_n)}{n^2} -\rho_n [R(\mathbf{Z}_n) S R(\mathbf{Z}_n)^\intercal \mathbf 1]_a \right| \Biggr).
\end{align*}
Applying Theorem \ref{thm_7.4} to each term, we get
\begin{align*}
\left| X(\mathbf{e}_n,\mathbf{Z}_{n})- X(\mathbf{Z}_{n},\mathbf{Z}_{n}) \right|&\leq M \Biggl(  \sum_{a,b} \sqrt{\frac{8 s_\text{max} \rho_n  \log k}{n^2}} +  \sum_{a,b} \sqrt{\frac{8 s_\text{max} \rho_n  \log k}{n^2}} \\
& \qquad + 2  \sum_{a}k \sqrt{\frac{8 s_\text{max} \rho_n  \log k}{n^2}} + 2  \sum_{a} k \sqrt{\frac{8 s_\text{max} \rho_n  \log k}{n^2}} \Biggr) \\
& \leq 6 M k^2 \rho_n \sqrt{\frac{8 s_\text{max} \log k}{\rho_n n^2}} .
\end{align*}
Since $n\rho_n \to \infty$, the term $\sqrt{\frac{8 s_{\max} \log k}{\rho_n n^2}}$ vanishes as $n \to \infty$. Thus, for any $\tilde \delta > 0$, there exists $n_0$ such that for all $n > n_0$ 
\[
|X(\mathbf{e}_n, \mathbf{Z}_n) - X(\mathbf{Z}_n, \mathbf{Z}_n)| \leq \tilde \delta \rho_n,
\]
simultaneously for all $\mathbf{e}_n \in \mathcal{F}(n, \alpha)$, eventually almost surely.

\section{Proof of auxiliary results for Theorem \ref{thm_strong} (strong consistency)} \label{proof_aux_strong}

\subsection{Proof of Proposition \ref{prop7.9}}\label{proof_prop7.9}
Recall that
\begin{align*} 
W&(\mathbf{e}_n,\mathbf{z}_n)  \\
&= \sum_{1\leq a, b\leq k} W_{ab}(\mathbf{e}_n,\mathbf{z}_n) \left[ \log \Bigl( \rho_n[R(\mathbf{z}_{n},\mathbf{z}_{n})SR(\mathbf{z}_{n},\mathbf{z}_{n})]_{ab }\Bigr) -  2     \log  \Bigl( \rho_n[R(\mathbf{z}_{n},\mathbf{z}_{n})SR(\mathbf{z}_{n},\mathbf{z}_{n})\mathbf 1]_{a }\Bigr) \right], 
\end{align*}
and
$$W_{ab}(\mathbf{e}_n,\mathbf{z}_n) =\frac{1}{n^2} \left\{ o_{ab}(\mathbf{e}_n) - \mathbb{E}[o_{ab}(\mathbf{e}_n)\mid \mathbf{z}_n] - o_{ab}(\mathbf{z}_n) + \mathbb{E}[o_{ab}(\mathbf{z}_n)\mid \mathbf{z}_n] \right\} .$$
We then can write 
\begin{align*} 
W&(\mathbf{e}_n,\mathbf{z}_n) =  \frac 1 {n^2} \sum_{1\leq a, b\leq k}  \sum_{1\leq i,j\leq n} Y_{ij}^{ab} - \mathbb{E} \left[Y_{ij}^{ab} \mid  \mathbf{z}_n \right],\end{align*}
with 
\begin{align*} Y_{ij}^{ab} = &\left[ \log \Bigl( \rho_n[R(\mathbf{z}_{n},\mathbf{z}_{n})SR(\mathbf{z}_{n},\mathbf{z}_{n})]_{ab }\Bigr) -  2     \log  \Bigl( \rho_n[R(\mathbf{z}_{n},\mathbf{z}_{n})SR(\mathbf{z}_{n},\mathbf{z}_{n})\mathbf 1]_{a }\Bigr) \right] \\
&\times \left[ A_{ij} \1_{e_i=a,e_j=b}- A_{ij} \1_{z_i=a,z_j=b}\right].
\end{align*}
Consider 
$$Y_{ij}= \sum_{1\leq a, b\leq k} Y_{ij}^{ab},$$
then we can write 
$$Y_{ij}=\left[ \log \Bigl( \frac{[R(\mathbf{z}_{n},\mathbf{z}_{n})SR(\mathbf{z}_{n},\mathbf{z}_{n})]_{e_ie_j } }{[R(\mathbf{z}_{n},\mathbf{z}_{n})SR(\mathbf{z}_{n},\mathbf{z}_{n})]_{z_iz_j }}\Bigr) -  2     \log  \Bigl( \frac{[R(\mathbf{z}_{n},\mathbf{z}_{n})SR(\mathbf{z}_{n},\mathbf{z}_{n})\mathbf 1]_{e_i }}{[R(\mathbf{z}_{n},\mathbf{z}_{n})SR(\mathbf{z}_{n},\mathbf{z}_{n})\mathbf 1]_{z_i }}\Bigr) \right] A_{ij}, $$
and we obtain
\begin{align*} 
W&(\mathbf{e}_n,\mathbf{z}_n) =  \frac 1 {n^2}  \sum_{1\leq i,j\leq n} Y_{ij} - \mathbb{E} \left[Y_{ij}\mid \mathbf{z}_n\right].\end{align*}
Let us denote 
$$\bar{S}_a(\mathbf{e}) = \sum_{b=1}^k \frac{n_b(\mathbf{e})}{n} S_{ab},$$ 
for any labeling $\mathbf{e} \in [k]^n$. By expanding the matrix products evaluated at the true latent positions $\mathbf{Z}_n$, we have 
$$\ [R(\mathbf{Z}_{n},\mathbf{Z}_{n}) S R(\mathbf{Z}_{n},\mathbf{Z}_{n})]_{e_i e_j} = \frac{n_{e_i}(\mathbf{Z}_n) n_{e_j}(\mathbf{Z}_n)}{n^2} S_{e_i e_j},$$ 
and 
$$ [R(\mathbf{Z}_{n},\mathbf{Z}_{n}) S R(\mathbf{Z}_{n},\mathbf{Z}_{n}) \mathbf{1}]_{e_i} = \frac{n_{e_i}(\mathbf{Z}_n)}{n} \bar{S}_{e_i}(\mathbf{Z}_n).$$
Conditioning on $\mathbf{Z}_n= \mathbf{z}_n $ and substituting these into the definition of $W(\mathbf{e}_n,\mathbf{z}_n)$, the logarithmic terms inside the brackets of $Y_{ij}$ become
\begin{align*}
&\log \left( \frac{n_{e_i}(\mathbf{z}_n) n_{e_j}(\mathbf{z}_n) S_{e_i e_j}}{n_{z_i}(\mathbf{z}_n) n_{z_j}(\mathbf{z}_n) S_{z_i z_j}} \right) - 2 \log \left( \frac{n_{e_i}(\mathbf{z}_n) \bar{S}_{e_i}(\mathbf{z}_n)}{n_{z_i}(\mathbf{z}_n) \bar{S}_{z_i}(\mathbf{z}_n)} \right) \\
&= \log \left( \frac{S_{e_i e_j}}{S_{z_i z_j}} \right) - 2 \log \left( \frac{\bar{S}_{e_i}(\mathbf{z}_n)}{\bar{S}_{z_i}(\mathbf{z}_n)} \right) + \log \left( \frac{n_{e_j}(\mathbf{z}_n)}{n_{z_j}(\mathbf{z}_n)} \right) - \log \left( \frac{n_{e_i}(\mathbf{z}_n)}{n_{z_i}(\mathbf{z}_n)} \right).
\end{align*}
Let $\lambda_{ij} = \mathbb{E}[A_{ij} \mid \mathbf{Z}_n = \mathbf{z}_n]$. Since $(A_{ij} - \lambda_{ij})$ is symmetric in $i,j$, the summation over $1 \leq i, j \leq n$ of the last two terms cancels out perfectly
\[
\sum_{i,j} (A_{ij} - \lambda_{ij}) \log \left( \frac{n_{e_j}(\mathbf{z}_n)}{n_{z_j}(\mathbf{z}_n)} \right) - \sum_{i,j} (A_{ij} - \lambda_{ij}) \log \left( \frac{n_{e_i}(\mathbf{z}_n)}{n_{z_i}(\mathbf{z}_n)} \right) = 0.
\]
Thus, we can equivalently rewrite the deviation as
\[
W(\mathbf{e}_n,\mathbf{z}_n) = \frac{1}{n^2} \sum_{1 \leq i, j \leq n} (A_{ij} - \lambda_{ij}) c_{ij}, \quad \text{where} \quad c_{ij} = \log \left( \frac{S_{e_i e_j}}{S_{z_i z_j}} \right) - 2 \log \left( \frac{\bar{S}_{e_i}(\mathbf{z}_n)}{\bar{S}_{z_i}(\mathbf{z}_n)} \right).
\]
Using the symmetry of $A$ and $c_{ij}$, we write the sum over independent variables 
$$n^2 W(\mathbf{e}_n,\mathbf{z}_n)  = 2 \sum_{i<j} (A_{ij} - \lambda_{ij}) c_{ij} + \sum_i (A_{ii} - \lambda_{ii}) c_{ii}.$$
Using the conditional Poisson moment generating function $\mathbb{E}[e^{u(A_{ij} - \lambda_{ij})}] = \exp(\lambda_{ij} L(u))$ where $L(u) = e^u - u - 1$, we obtain for $\tau>0$
\begin{align*}
\mathbb{E} [\exp\{\tau n^2 W(\mathbf{e}_n,\mathbf{z}_n) \} \mid \mathbf{z}_n] &= \prod_{i<j} \exp(\lambda_{ij} L(2\tau c_{ij})) \prod_{i} \exp(\lambda_{ii} L(\tau c_{ii})) \\
&\leq \exp \left( \sum_{1 \leq i \leq j \leq n} \lambda_{ij} L(2\tau c_{ij}) \right)\\
& \leq \exp \left( \sum_{1 \leq i, j \leq n} \lambda_{ij} L(2\tau c_{ij}) \right).
\end{align*}
To decouple the terms in $c_{ij} = c_{ij}^{(1)} - c_{ij}^{(2)}$, we use the convexity of $L$, which guarantees $L(u+v) \leq \frac{1}{2}L(2u) + \frac{1}{2}L(2v)$. Setting $t = 4\tau$, we have $L(2\tau c_{ij}) \leq \frac{1}{2}L(t c_{ij}^{(1)}) + \frac{1}{2}L(-t c_{ij}^{(2)})$. Hence, applying Chernoff's bound for $\tau > 0$, we have
\begin{align}\label{borne_P(w)}
\mathbb{P}&(W(\mathbf{e}_n,\mathbf{z}_n)  > x \mid \mathbf{z}_n) \nonumber \\
&\leq \inf_{t \in (0, 1]} \exp \left( -n^2 x \frac{t}{4} + \frac{1}{2} \sum_{i,j} \lambda_{ij} L\left( t \log \frac{S_{e_i e_j}}{S_{z_i z_j}} \right) + \frac{1}{2} \sum_{i,j} \lambda_{ij} L\left( 2t \log \frac{\bar{S}_{z_i}(\mathbf{z}_n)}{\bar{S}_{e_i}(\mathbf{z}_n)} \right) \right).
\end{align}
Define $K_t(p \| q) = p^t q^{1-t} - t q \log(p/q) - q$. One easily verifies that $$q L(t \log(p/q)) = K_t(p \| q).$$Replacing $\lambda_{ij} = \rho_n \theta_i \theta_j S_{z_i z_j}$, the first sum in the right-hand side of \eqref{borne_P(w)} evaluates to
\begin{align}\label{premierK}
\rho_n\sum_{i,j} \theta_i \theta_j S_{z_i z_j} &L\left( t \log \frac{S_{e_i e_j}}{S_{z_i z_j}} \right) \\
&= \rho_n\sum_{i,j} \theta_i \theta_j K_t(S_{e_i e_j} \| S_{z_i z_j}) \nonumber \\
&= \rho_nn^2 \sum_{a, a'} \sum_{b, b'} R_{aa'} K_t(S_{ab} \| S_{a'b'}) R_{bb'} \nonumber\\
& =2\rho_n n^2\;\sum_{b\neq b'}\sum_{a} R_{aa}K_t(S_{ab} \| S_{ab'})R_{bb'} + \rho_nn^2\;\sum_{a\neq a'}\sum_{b\neq b'} R_{aa'}K_t(S_{ab} \| S_{a'b'})R_{bb'},\nonumber
\end{align}
where $R = R(\mathbf{e}_n, \mathbf{z}_n)$. For the second sum in the right-hand side of \eqref{borne_P(w)}, using that $\sum_j \lambda_{ij} = \rho_n n \theta_i \bar{S}_{z_i}(\mathbf{z}_n)$, we obtain
\begin{align}\label{deuxK}
\sum_{i,j} \lambda_{ij} L\left( 2t \log \frac{\bar{S}_{z_i}(\mathbf{z}_n)}{\bar{S}_{e_i}(\mathbf{z}_n)} \right) &= \rho_n n \sum_i \theta_i \bar{S}_{z_i}(\mathbf{z}_n) L\left( 2t \log \frac{\bar{S}_{z_i}(\mathbf{z}_n)}{\bar{S}_{e_i}(\mathbf{z}_n)} \right) \nonumber \\
&= \rho_n n \sum_i \theta_i K_{2t}(\bar{S}_{e_i}(\mathbf{z}_n) \| \bar{S}_{z_i}(\mathbf{z}_n))\nonumber \\
&= \rho_n n^2 \sum_{a, a'} R_{aa'} K_{2t}(\bar{S}_a \| \bar{S}_{a'}).
\end{align}
Since $m/n = \sum_{a \neq a'} R_{a'a}$ and $K_t(S_{ab} \| S_{a'b'}) \leq D$  for all $a,b,a',b'\in[k]$ and all $t>0$, where $D$ is a constant depending only on $S$, the second sum in \eqref{premierK} is bounded by $\rho_nDm^2$. Similarly, bounding $K_{2t}(\bar{S}_a \| \bar{S}_{a'}) \leq \tilde{D}$, \eqref{deuxK} is bounded by $\rho_n n \tilde{D} m$.
Combining both tails via the union bound and using that $R_{aa} \leq[R^\intercal \mathbf 1]_a$, we conclude:
\begin{align*}
\mathbb{P}\Bigl(W(\mathbf{e}_n,\mathbf{z}_n) >x \;\Big|\; \mathbf{Z}_n=\mathbf{z}_n\Bigr)& \;\leq\;  \exp\Bigl[ - \sup_{t\in (0,1 ]}  \Bigl\{n^2 x \frac{t}{4} -  \rho_n n^2 C_t(R,S)\Bigr\}  + \tilde D \rho_n n m + D\rho_n m^2 \Bigr],
\end{align*}
with $C_t(R,S)$ given by \eqref{tildeC}.

\subsection{Proof of Proposition \ref{propB.2}} \label{proof_propB.2}

Let 
$$
Y^{ab} = \sum_{1\leq i\leq j\leq n} Y^{ab}_{ij}
$$
with  
\begin{align*}
Y_{ij}^{ab} = \frac{1}{n^2}&\Bigl( \mathds{1}\{e_i=a,e_j=b\}+ \mathds{1}\{e_i=b,e_j=a\} 
- \mathds{1}\{z_i=a,z_j=b\} - \mathds{1}\{z_i=b,z_j=a\} \Bigr)A_{ij}\,.
\end{align*}
Conditioned on $\mathbf{Z}_n=\mathbf{z}_n $, the variables $Y_{ij}^{ab}$, for $1\leq i\leq j\leq n$ are independent. 
Applying Chernoff's bound, we obtain
\begin{align*}
 \mathbb{P}( Y^{ab} -  \mathbb{E}(Y^{ab}) > x \mid \mathbf{Z}_n=\mathbf{z}_n) \leq  \inf_{t>0} \exp\bigl(- xt  \bigr)\prod_{1\leq i \leq j\leq n}\mathbb{E}\bigl( \exp[ t(Y_{ij}^{ab}-\mathbb{E}(Y_{ij}^{ab}))]\bigr),
\end{align*}
where the expectations are given $\mathbf{Z}_n=\mathbf{z}_n $.
Let's analyze the expected values. Define
\begin{align*}
\mathds{1}_{ij}^{ab}(\mathbf{e}_n,\mathbf{z}_n) \;=\; & \mathds{1}\{e_i=a,e_j=b\}+ \mathds{1}\{e_i=b,e_j=a\} - \mathds{1}\{z_i=a,z_j=b\} - \mathds{1}\{z_i=b,z_j=a\}\,.
\end{align*}
We have
$$
\mathds{E}(Y_{ij}^{ab} \mid \mathbf{Z}_n=\mathbf{z}_n ) = \frac{\rho_n\theta_i \theta_jS_{z_iz_j}\mathds{1}_{ij}^{ab}(\mathbf{e}_n,\mathbf{z}_n)}{n^2}\,.
$$
On the other hand
\begin{align*}
\mathds{E}(\exp(tY_{ij}^{ab}) \mid \mathbf{Z}_n=\mathbf{z}_n ) \;&=\;  \exp\left(\theta_i \theta_j \rho_nS_{z_iz_j}\exp\left[\frac{t\mathds{1}_{ij}^{ab}(\mathbf{e}_n,\mathbf{z}_n)}{n^2}-1\right]\right).
\end{align*}
Then 
\begin{align*}
\mathds{E}(&\exp(t(Y_{ij}^{ab}- \mathds{E}(Y_{ij}^{ab}))\,|\, \mathbf{Z}_n=\mathbf{z}_{n}) \\
&= \exp\left(\theta_i \theta_j \rho_nS_{z_iz_j}\exp\left[\frac{t\mathds{1}_{ij}^{ab}(\mathbf{e}_n,\mathbf{z}_n)}{n^2}-1\right]\right) \exp \left( - \frac{t\theta_i \theta_j\rho_nS^t_{z_iz_j}\mathds{1}_{ij}^{ab}(\mathbf{e}_n,\mathbf{z}_n)}{n^2} \right)
\\&=\exp\left( \theta_i \theta_j \rho_n S_{z_iz_j} L\left[ \frac{t\mathds{1}_{ij}^{ab}(\mathbf{e}_n,\mathbf{z}_n)}{n^2}\right] \right),
\end{align*}
with $L(y)=e^y-y-1$.
By choosing $t=n^2$ and using the fact  that $L(y)\leq 3|y|$ for $y\in[-2,2]$, we find that
\begin{align*}
\mathbb{P}( Y^{ab} - \mathbb{E}(Y^{ab})  > x\mid \mathbf{Z}_n=\mathbf{z}_n) 
&\leq\; \exp \bigl(- xn^2\bigr)\prod_{1\leq i\le j\leq n}\exp\Bigl[3 \theta_i \theta_j \rho_nS_{z_iz_j}|\mathds{1}^{ab}_{ij}(\mathbf{e}_n,\mathbf{z}_n)|\Bigr]\\
&=\; \exp \bigl(- xn^2\bigr)\exp\Bigl[3\rho_n \sum_{1\leq i\le j\leq n} \theta_i \theta_j S_{z_iz_j} |\mathds{1}^{ab}_{ij}(\mathbf{e}_n,\mathbf{z}_n)|\Bigr].
\end{align*}
By the definition of $\mathds{1}^{ab}_{ij}(\mathbf{e}_n,\mathbf{z}_n)$, we can bound its absolute value by $ 2 ( \mathds{1}\{e_i \neq z_i\}+  \mathds{1}\{e_j \neq z_j\} )$. So we have 
\begin{align*}
 \sum_{1\leq i\le j\leq n} \theta_i \theta_j S_{z_iz_j} |\mathds{1}^{ab}_{ij}(\mathbf{e}_n,\mathbf{z}_n)|
&\leq\;2 s_{\max}  \sum_{1\leq i\le j\leq n} \theta_i \theta_j ( \mathds{1}\{e_i \neq z_i\}+  \mathds{1}\{e_j \neq z_j\} )\\
&\leq\;4 s_{\max}  \sum_{1\leq i, j\leq n} \theta_i \theta_j  \mathds{1}\{e_i \neq z_i\} \\
&\leq\;4 s_{\max}  \sum_{1\leq i \leq n} \theta_i \mathds{1}\{e_i \neq z_i\} \left( \sum_{1\leq j \leq n} \theta_j  \right)\\
&\leq\;4 s_{\max} m(\mathbf{e}_n,\mathbf{z}_n)n.
\end{align*}
Then, we find that
\begin{align*}
\mathbb{P}( Y^{ab} - \mathbb{E}(Y^{ab})  > x\mid \mathbf{Z}_n=\mathbf{z}_n) 
&\leq\; \exp \bigl(- xn^2 + 12\rho_ns_{\max} m(\mathbf{e}_n,\mathbf{z}_n)n \bigr)\,.
\end{align*}
An analogous bound can be derived for the left deviations. Therefore,
\begin{equation*}
\mathbb{P}( |W_{ab}(\mathbf{e}_n,\mathbf{z}_n)| > x\,\mid \mathbf{Z}_n=\mathbf{z}_n) \;\leq\;  2\exp\bigl(- xn^2 + 12\rho_ns_{\max} m(\mathbf{e}_n,\mathbf{z}_n)n\bigr)\,.
\end{equation*}

\subsection{Proof of Proposition \ref{prop7.7}}\label{proof_prop7.7}

Applying the concentration inequality from Proposition \ref{propB.2} and a union bound over the number of differences $m$ and the community assignments $\mathbf{e}_n \in [k]^n$ such that $m(\mathbf{e}_n, \mathbf{z}_n) = m$, and using the standard bound $\binom{n}{m} \leq n^m$, we obtain
\begin{align}\label{ineq_W}
\mathbb{P}\Big( 
\sup_{\substack{a,b \in [k],\, m \in [n],\\ \mathbf{e}_n \in [k]^n:\, m(\mathbf{e}_n,\mathbf{z}_n)=m}} 
&|W_{ab}(\mathbf{e}_n,\mathbf{z}_n)| > x 
\,\Big|\, \mathbf{Z}_n=\mathbf{z}_n \Big)
\nonumber\\&\leq 2 \sum_{a,b \in [k]}\sum_{m=1}^n \binom{n}{m} (k-1)^m 
\exp\bigl(- xn^2 + 12\rho_n s_{\max} mn\bigr) \nonumber \\
&\leq 2k^2 \sum_{m=1}^n 
\exp\bigl(- xn^2 + 12\rho_n s_{\max} mn + m \log[(k-1)n]\bigr).
\end{align}
By setting  $x = c \rho_n m/n$, the probability above can be bounded as follows 
\begin{align*}
\mathbb{P}( \; \sup_{a,b,m,\mathbf{e}_n} |W_{ab}(\mathbf{e}_n,\mathbf{z}_n)| > x\mid \mathbf{Z}_n=\mathbf{z}_n) \;&\leq 2k^2 \sum_{m=1}^n\exp\bigl(- m(c\rho_n n - 12\rho_ns_{\max} n-\log [(k-1)n])).
\end{align*}
Given $\rho_n \geq \log n / n$ and choosing $c > 6s_{\max}$ sufficiently large, the series above is summable over $n$. By the Borel-Cantelli lemma, it follows that
$$
 |W_{ab}(\mathbf{e}_n,\mathbf{z}_n) | \leq  \frac{c \rho_n m(\mathbf{e}_n,\mathbf{z}_n)}{n},
$$
simultaneously for all $a,b \in [k]$ and $\mathbf{e}_n \in[k]^n$, eventually almost surely as $n \to \infty$.
This proves the first statement. To prove the second, we add and subtract the corresponding conditional expectations $\mathbb{E}(o_{ab}(\mathbf{e}_n) \mid  \mathbf{z}_n)$ and $\mathbb{E}(o_{ab}(\mathbf{z}_n) \mid  \mathbf{z}_n)$. Applying the triangle inequality and combining \eqref{ineq_W} with the fact that
\begin{align*}
 \left|\frac{\mathbb E(o_{ab}(\mathbf{e}_n)\mid \mathbf{z}_n)}{n^2} - \frac{\mathbb E(o_{ab}(\mathbf{z}_n)\mid \mathbf{z}_n)}{n^2}\right| \leq \frac{2 s_{\max} \rho_n m(\mathbf{e}_n,\mathbf{z}_n)}{n},
\end{align*}
we obtain 
\begin{align*}
 \left|\frac{o_{ab}(\mathbf{e}_n)}{n^2} - \frac{o_{ab}(\mathbf{z}_n)}{n^2}\right|  &=  \left|W_{ab}(\mathbf{e}_n,\mathbf{z}_n)  +\frac{\mathbb E(o_{ab}(\mathbf{e}_n)\mid \mathbf{z}_n)}{n^2} - \frac{\mathbb E(o_{ab}(\mathbf{z}_n)\mid \mathbf{z}_n)}{n^2} \right|\\
&  \leq  \left|W_{ab}(\mathbf{e}_n,\mathbf{z}_n)  \right| + \left|\frac{\mathbb E(o_{ab}(\mathbf{e}_n)\mid \mathbf{z}_n)}{n^2} - \frac{\mathbb E(o_{ab}(\mathbf{z}_n)\mid \mathbf{z}_n)}{n^2}\right|\\
& \leq  \frac{c \rho_n m(\mathbf{e}_n,\mathbf{z}_n)}{n} +\frac{2 s_{\max} \rho_n m(\mathbf{e}_n,\mathbf{z}_n)}{n}\\
&\leq  \frac{c' \rho_n m(\mathbf{e}_n,\mathbf{z}_n)}{n},
\end{align*}
with $c'>0$ a given constant only depending on $S$.

\subsection{Proof of Theorem \ref{thm7.10}}\label{proof_thm7.10}
From now on, we will assume that Proposition \ref{prop7.7} holds, as this happens simultaneously for all $a,b\in [k]$ and all $\mathbf{e}_n \in [k]^n$ eventually almost surely as $n \to \infty$. Therefore, all subsequent computations are carried out under the event $\mathcal C$ defined in \eqref{def_C} and assuming the bounds in Proposition  \ref{prop7.7}.
Observe that 
\begin{align*}
X(\mathbf{e}_n,\mathbf{Z}_{n})- X(\mathbf{Z}_{n},\mathbf{Z}_{n}) 
&= \sum_{a,b} \Biggl[ f\Bigl(\frac{o_{ab}(\mathbf{e}_n)}{n^2}\Bigr) - f\Bigl(\frac{\mathbb{E}[o_{ab}(\mathbf{e}_n)\mid \mathbf{Z}_n]}{n^2}\Bigr) \Biggr] \\
&\quad - 2 \sum_a \Biggl[ f\Bigl(\frac{D_a(\mathbf{e}_n)}{n^2}\Bigr) - f\Bigl(\frac{\mathbb{E}[D_a(\mathbf{e}_n)\mid \mathbf{Z}_n]}{n^2}\Bigr) \Biggr] \\
&\quad - \sum_{a,b} \Biggl[ f\Bigl(\frac{o_{ab}(\mathbf{Z}_n)}{n^2}\Bigr) - f\Bigl(\frac{\mathbb{E}[o_{ab}(\mathbf{Z}_n)\mid \mathbf{Z}_n]}{n^2}\Bigr) \Biggr] \\
&\quad + 2 \sum_a \Biggl[ f\Bigl(\frac{D_a(\mathbf{Z}_n)}{n^2}\Bigr) - f\Bigl(\frac{\mathbb{E}[D_a(\mathbf{Z}_n)\mid \mathbf{Z}_n]}{n^2}\Bigr) \Biggr],
\end{align*}
with $
f(x) = x\log x .$
Since $KL(x||y)=x \log \frac xy +y-x \geq 0$ we have, for all $x,y>0$, that 
\[
f(x) \geq x\log y + x-y\,.
\]
We use this inequality for $x=\mathbb{E}[o_{ab}(\mathbf{e}_n)\mid \mathbf{Z}_n]/n^2$ and $y=o_{ab}(\mathbf{e}_n)/n^2$ and we obtain that
\[
f\Bigl(\frac{\mathbb{E}[o_{ab}(\mathbf{e}_n)\mid \mathbf{Z}_n]}{n^2}\Bigr) 
\;\geq\; \frac{\mathbb{E}[o_{ab}(\mathbf{e}_n)\mid \mathbf{Z}_n]}{n^2} 
\log \frac{o_{ab}(\mathbf{e}_n)}{n^2} + \frac{\mathbb{E}[o_{ab}(\mathbf{e}_n)\mid \mathbf{Z}_n]}{n^2} - \frac{o_{ab}(\mathbf{e}_n)}{n^2}.
\]
Similarly, we bound 
\begin{align*}
f\Bigl(\frac{o_{ab}(\mathbf{Z}_n)}{n^2}\Bigr) 
&\geq\; \frac{o_{ab}(\mathbf{Z}_n)}{n^2} \log \frac{\mathbb{E}[o_{ab}(\mathbf{Z}_n)\mid \mathbf{Z}_n]}{n^2} + \frac{o_{ab}(\mathbf{Z}_n)}{n^2} - \frac{\mathbb{E}[o_{ab}(\mathbf{Z}_n)\mid \mathbf{Z}_n]}{n^2},\\
 f\Bigl(\frac{D_a(\mathbf{e}_n)}{n^2}\Bigr) 
&\geq\; \frac{D_a(\mathbf{e}_n)}{n^2} \log \frac{\mathbb{E}[D_a(\mathbf{e}_n)\mid \mathbf{Z}_n]}{n^2} + \frac{D_a(\mathbf{e}_n)}{n^2} - \frac{\mathbb{E}[D_a(\mathbf{e}_n)\mid \mathbf{Z}_n]}{n^2},\\
 f\Bigl(\frac{\mathbb{E}[D_a(\mathbf{Z}_n)\mid \mathbf{Z}_n]}{n^2}\Bigr) 
&\geq\; \frac{\mathbb{E}[D_a(\mathbf{Z}_n)\mid \mathbf{Z}_n]}{n^2} \log \frac{D_a(\mathbf{Z}_n)}{n^2} + \frac{\mathbb{E}[D_a(\mathbf{Z}_n)\mid \mathbf{Z}_n]}{n^2} - \frac{D_a(\mathbf{Z}_n)}{n^2}.
\end{align*}
By an abuse of notation, all the expectations below are conditionally on $\mathbf{Z}_n$, without any further notice.
Using this, we obtain the bound 
\begin{align*}
X(\mathbf{e}_n,\mathbf{Z}_{n})- X(\mathbf{Z}_{n},\mathbf{Z}_{n}) \;\leq\;& \sum_{1\leq a, b\leq k}\, \Bigl(\dfrac{o_{ab}(\mathbf{e}_n)}{n^2}- 
\mathbb{E}\Bigl(\dfrac{o_{ab}(\mathbf{e}_n)}{n^2}\Bigr)\Bigr)\left[  \log \Bigl(\frac{o_{ab}(\mathbf{e}_n)}{n^2}\Bigr) +1 \right]\\
&+  2  \sum_{1\leq a\leq k}\, \Bigl( \mathbb{E}\Bigl[\dfrac{D_{a}(\mathbf{e}_n)}{n^2}\Bigr]  -  \dfrac{D_{a}(\mathbf{e}_n)}{n^2}    \Bigr) \Bigl[ \log \Bigl(\frac{\mathbb{E}\Bigl[D_{a}(\mathbf{e}_n)\Bigr]}{n^2}\Bigr)+1 \Bigr]\\
&+ \sum_{1\leq a,b\leq k}\, \Bigl( \mathbb{E}\Bigl[\dfrac{o_{ab}(\mathbf{Z}_n)}{n^2}\Bigr]  -  \dfrac{o_{ab}(\mathbf{Z}_n)}{n^2}    \Bigr) \Bigl[ \log \Bigl(\frac{\mathbb{E}\Bigl[o_{ab}(\mathbf{Z}_n)\Bigr]}{n^2}\Bigr) +1 \Bigr]\\
&+ 2 \sum_{1\leq a\leq k}\, \Bigl(\dfrac{D_{a}(\mathbf{Z}_n)}{n^2}- 
\mathbb{E}\Bigl(\dfrac{D_{a}(\mathbf{Z}_n)}{n^2}\Bigr)\Bigr) \left[ \log \Bigl(\frac{D_{a}(\mathbf{Z}_n)}{n^2}\Bigr) +1 \right].
\end{align*}
Recall that $\sum_{ab} o_{ab}(\mathbf{e}_n)= \sum_{ab} o_{ab}(\mathbf{Z}_n)$ is a constant, and $D_{a}(\mathbf{e}_n)= \sum_{b} o_{ab}(\mathbf{e}_n)$. We can thus simplify the bound by 
\begin{align*}
X(\mathbf{e}_n,\mathbf{Z}_{n})- X(\mathbf{Z}_{n},\mathbf{Z}_{n}) \;\leq\;& \sum_{1\leq a, b\leq k}\, \Bigl(\dfrac{o_{ab}(\mathbf{e}_n)}{n^2}- 
\mathbb{E}\Bigl(\dfrac{o_{ab}(\mathbf{e}_n)}{n^2}\Bigr)\Bigr) \log \Bigl(\frac{o_{ab}(\mathbf{e}_n)}{n^2}\Bigr) \\
&+  2  \sum_{1\leq a\leq k}\, \Bigl( \mathbb{E}\Bigl[\dfrac{D_{a}(\mathbf{e}_n)}{n^2}\Bigr]  -  \dfrac{D_{a}(\mathbf{e}_n)}{n^2}    \Bigr) \log \left(\frac{\mathbb{E}\Bigl[D_{a}(\mathbf{e}_n)\Bigr]}{n^2}\right)\\
&+ \sum_{1\leq a,b\leq k}\, \Bigl( \mathbb{E}\Bigl[\dfrac{o_{ab}(\mathbf{Z}_n)}{n^2}\Bigr]  -  \dfrac{o_{ab}(\mathbf{Z}_n)}{n^2}    \Bigr)\log \left(\frac{\mathbb{E}\Bigl[o_{ab}(\mathbf{Z}_n)\Bigr]}{n^2}\right) \\
&+ 2 \sum_{1\leq a\leq k}\, \Bigl(\dfrac{D_{a}(\mathbf{Z}_n)}{n^2}- 
\mathbb{E}\Bigl(\dfrac{D_{a}(\mathbf{Z}_n)}{n^2}\Bigr)\Bigr) \log \Bigl(\frac{D_{a}(\mathbf{Z}_n)}{n^2}\Bigr).
\end{align*}
Observe that, by denoting $R=R(\mathbf{Z}_{n},\mathbf{Z}_{n})$, we can write 
$$\frac{1}{n^2} \mathbb{E}\Bigl[o_{ab}(\mathbf{Z}_n)\Bigr]=  \rho_n[RSR]_{ab }, $$and $$\frac{1}{n^2}\mathbb{E}\Bigl[D_a(\mathbf{Z}_n)\Bigr] = \rho_n [RSR\mathbf 1]_{a }. $$
We also can rewrite
\begin{align*}
\log \Bigl(\dfrac{o_{ab}(\mathbf{e}_n)}{n^2}\Bigr) \;=\; \log \Bigl[\rho_n [RSR]_{ab} \Bigl(1 + \frac{ o_{ab}(\mathbf{e}_n)/n^2 - \rho_n[RSR]_{ab}}{\rho_n[RSR]_{ab}}\Bigr)\Bigr],
\end{align*}
and similarly 
\begin{align*}
\log  \Bigl(\dfrac{\mathbb E [D_{a}(\mathbf{e}_n)]}{n^2}\Bigr)  = \log \Bigl[ \rho_n[RSR\mathbf 1]_{a} \Bigl(1 + \frac{ \mathbb E [D_{a}(\mathbf{e}_n)]/n^2 -\rho_n [RSR\mathbf 1]_{a}}{\rho_n[RSR\mathbf 1]_{a}}\Bigr)\Bigr].
\end{align*}
We then obtain, with $W_{ab}(\mathbf{e}_n,\mathbf{z}_n) $ defined in \eqref{def_W}, that 
\begin{equation}\label{XX}
\begin{split}
X(\mathbf{e}_n,\mathbf{Z}_{n})-& X(\mathbf{Z}_{n},\mathbf{Z}_{n}) \\&\;\leq\; \sum_{1\leq a, b\leq k}\, W_{ab}(\mathbf{e}_n,\mathbf{z}_n) \log \Bigl(\rho_n [RSR]_{ab }\Bigr) -  2  \sum_{1\leq a,b \leq k}\,  W_{ab}(\mathbf{e}_n,\mathbf{z}_n)    \log  \Bigl(\rho_n [RSR\mathbf 1]_{a }\Bigr)  \\
&+  \sum_{1\leq a, b\leq k}\, \Bigl(\dfrac{o_{ab}(\mathbf{e}_n)}{n^2}- \mathbb{E}\Bigl(\dfrac{o_{ab}(\mathbf{e}_n)}{n^2}\Bigr)\Bigr) \log \Bigl(  1 + \frac{ o_{ab}(\mathbf{e}_n)/n^2 -\rho_n [RSR]_{ab}}{\rho_n[RSR]_{ab}} \Bigr) \\
&+ 2 \sum_{1\leq a\leq k}\, \Bigl( \mathbb{E}\Bigl[\dfrac{D_{a}(\mathbf{e}_n)}{n^2}\Bigr]  -  \dfrac{D_{a}(\mathbf{e}_n)}{n^2}    \Bigr) \log \Bigl(   1 + \frac{ \mathbb E[D_{a}(\mathbf{e}_n)]/n^2 -\rho_n [RSR\mathbf 1]_{a}}{\rho_n[RSR\mathbf 1]_{a}} \Bigr).
\end{split}
\end{equation}
Let us control the arguments of the logarithms in the last two terms. Recall that $R=R(\mathbf{Z}_{n},\mathbf{Z}_{n})$. On the event $\mathcal C$, using the triangle inequality and Proposition~\ref{prop7.7}, we have:
\begin{align}\label{borne_W_2}
\left| \frac{ o_{ab}(\mathbf{e}_n)}{n^2} -\rho_n [RSR]_{ab} \right|& \leq   \left|\frac{o_{ab}(\mathbf{e}_n)}{n^2} - \frac{o_{ab}(\mathbf{Z}_n)}{n^2}\right|  + \left| \frac{o_{ab}(\mathbf Z_n)}{n^2} - \rho_n [R(\mathbf{Z}_{n},\mathbf{Z}_{n})SR(\mathbf{Z}_{n},\mathbf{Z}_{n})]_{ab}  \right| \nonumber \\
& \leq \frac{c\rho_n m}{n}  + \sqrt{\frac{8 s_\text{max} \rho_n  \log k}{n^2}}.
\end{align}
Since $m \leq \epsilon n$ and $\rho_n \geq \log n/n$, this upper bound is $(c\epsilon + o(1))\rho_n$. Given that $\alpha < \min_a \pi_a$, the denominator in the logarithm satisfies $\rho_n [RSR]_{ab} \geq \alpha^2 \rho_n s_{\min}$. Thus, the relative variation in the log argument is bounded by $c'\epsilon /(\alpha^2 s_{\min})$. Similarly, the deterministic relative variation in the last logarithm is bounded by $c''\epsilon/(\alpha^2 s_{\min})$. By choosing $\epsilon > 0$ sufficiently small such that 
$$\frac{c'\epsilon}{\alpha^2 s_{\min}} \leq \frac{1}{2} \quad \text{and} \frac{c''\epsilon}{\alpha^2 s_{\min}} \leq \frac{1}{2},$$ these arguments are bounded in absolute value by $1/2$ for all sufficiently large $n$. Since the function $x \mapsto \log(1+x)$ is Lipschitz continuous on $[-1/2, 1/2]$ with $\log(1+x) \leq 2|x|$, we can safely bound the product terms using $|A\log(1+X) |\leq 2|A||X|$.
The difference in \eqref{XX} is thus upper bounded by $W(\mathbf{e}_n,\mathbf{z}_n) + \widetilde W(\mathbf{e}_n,\mathbf{z}_n) $ with
\begin{align*} 
W(\mathbf{e}_n,\mathbf{z}_n)  = &\sum_{1\leq a, b\leq k}\, W_{ab}(\mathbf{e}_n,\mathbf{z}_n) \log \Bigl( \rho_n[RSR]_{ab }\Bigr) -  2  \sum_{1\leq a,b \leq k}\,  W_{ab}(\mathbf{e}_n,\mathbf{z}_n)    \log  \Bigl(\rho_n [RSR\mathbf 1]_{a }\Bigr),  
\end{align*}
and 
\begin{align*} 
\widetilde  W(\mathbf{e}_n,\mathbf{z}_n)  = & 2\sum_{1\leq a, b\leq k}\, \Bigl| \dfrac{o_{ab}(\mathbf{e}_n)}{n^2}- \mathbb{E}\Bigl(\dfrac{o_{ab}(\mathbf{e}_n)}{n^2}\Bigr)\Bigr| \left| \frac{ o_{ab}(\mathbf{e}_n)}{n^2} -\rho_n [RSR]_{ab} \right|  \frac{ 1}{\rho_n[RSR]_{ab}}\\
&+ 4 \sum_{1\leq a\leq k}\, \Bigl|\dfrac{D_{a}(\mathbf{e}_n)}{n^2}- \mathbb{E}\Bigl(\dfrac{D_{a}(\mathbf{e}_n)}{n^2}\Bigr)\Bigr|  \left| \frac{ \mathbb{E}[D_{a}(\mathbf{e}_n)]}{n^2} -\rho_n [RSR\mathbf 1]_{a} \right|    \frac{ 1}{\rho_n[RSR\mathbf 1]_{a}}.
\end{align*}
We now focus on bounding $ \widetilde W(\mathbf{e}_n,\mathbf{z}_n)$. 
On $\mathcal C$ we have the bound
\begin{align}\label{borne_W_1}
\Bigl| \dfrac{o_{ab}(\mathbf{e}_n)}{n^2}- \mathbb{E}\Bigl(\dfrac{o_{ab}(\mathbf{e}_n)}{n^2}\Bigr)\Bigr| 
\leq \sqrt{\frac{8 s_\text{max} \rho_n  \log k}{n^2}},
\end{align}
simultaneously for all $a,b \in [k]$.
Then, since $ D_{a}(\mathbf{e}_n)= \sum_b o_{ab}(\mathbf{e}_n)$, the bounds \eqref{borne_W_2} and \eqref{borne_W_1} imply that
\begin{align*} 
\widetilde  W(\mathbf{e}_n,\mathbf{z}_n)  \leq &\, 2\sum_{1\leq a, b\leq k}\,  \sqrt{\frac{8 s_\text{max} \rho_n  \log k}{n^2}} \left( \frac{c\rho_n m(\mathbf{e}_n,\mathbf{z}_n)}{n}  + \sqrt{\frac{8 s_\text{max} \rho_n  \log k}{n^2}} \right)  \frac{ 1}{\alpha^2 \rho_n s_{\min}}\\
&+ 4 \sum_{1\leq a\leq k}\,  k \sqrt{\frac{8 s_\text{max} \rho_n  \log k}{n^2}}\left( \frac{kc\rho_n m(\mathbf{e}_n,\mathbf{z}_n)}{n}  +k \sqrt{\frac{8 s_\text{max} \rho_n  \log k}{n^2}} \right)   \frac{ 1}{k\alpha^2 \rho_n s_{\min}}\\
\leq & 6 k^2 \left( \frac{c''\rho_n m(\mathbf{e}_n,\mathbf{z}_n)}{n}  +  \sqrt{\frac{8 s_\text{max} \rho_n  \log k}{n^2}}\right)^2 \frac{ 1}{\alpha^2 \rho_n s_{\min}},
\end{align*}
with $c''$ a constant. Since $(a+b)^2 \leq 2a^2 +2b^2$, we further obtain
\begin{align*} 
\widetilde  W(\mathbf{e}_n,\mathbf{z}_n)  \leq & \frac{ 6 k^2}{\alpha^2 \rho_n s_{\min}} \left[ \left( \frac{c''\rho_n m(\mathbf{e}_n,\mathbf{z}_n)}{n} \right)^2 + \left( \sqrt{\frac{8 s_\text{max} \rho_n  \log k}{n^2}}\right)^2 \right]\\
\leq & C \frac{\rho_n m^2+1}{n^2 },
\end{align*}
where $C$ a constant only depending on $S$.
Then, by Proposition~\ref{prop7.9}, we have  that 
\begin{align*}
\mathbb P \left(  \left\{ X(\mathbf{e}_n,\mathbf{z}_n)  - X(\mathbf{z}_n,\mathbf{z}_n) > x \right\} \cap \mathcal C \mid \mathbf Z_n= \mathbf z_n\right)
&\leq \mathbb P \left( W(\mathbf{e}_n,\mathbf{z}_n) > x - \widetilde W (\mathbf{e}_n,\mathbf{z}_n)  \mid \mathbf Z_n= \mathbf z_n\right)\\
&\leq \exp\Bigl[ - \sup_{t\in (0,1 ]}  \Bigl\{n^2 \bigl[ xt-  \rho_n  C_t(R,S) \bigr]\Bigr\}  +\phi(n,m) \Bigr] 
\end{align*}
with $\phi(n,m) = C(\rho_n m^2 +\rho_n n m +1)$
and $C$ a constant only depending only on $S$ and $k$.

\subsection{Proof of Lemma \ref{lemma_der}}\label{proof_lemma_der}

Let us first compute the derivative of $ H_{S,n}(R)$ with respect to $R$.
Let $Q = RSR^\intercal$. The function can be written as 
$$ H_{S,n}(R) = \sum_{a,b}\rho_n Q_{ab} \log \left( \frac{Q_{ab}}{\rho_n n^2 q_a q_b} \right),$$
where $q_a = [Q \veco]_a = \sum_b Q_{ab}$ is the sum of the $a$-th row of $Q$.
We first differentiate $Q_{ab}$ with respect to $R_{uv}$.
By definition, $Q_{ab} = \sum_{k,l} R_{ak} S_{kl} R_{bl}$. Differentiating with respect to a specific entry $R_{uv}$ 
$$\frac{\partial Q_{ab}}{\partial R_{uv}} = \delta_{au} \sum_l S_{vl} R_{bl} + \delta_{bu} \sum_k R_{ak} S_{kv}.$$
Let us evaluate this derivative at the point $R = D = \tDiag(r)$, where $R_{ij} = r_i$ if $i=j$, and $0$ otherwise.
$$\frac{\partial Q_{ab}}{\partial R_{uv}} \bigg|_{R=D} = \delta_{au} S_{vb} r_b + \delta_{bu}r_a S_{av}.$$
We also have that
$$\frac{\partial  H_{S,n}}{\partial R_{uv}} = \underbrace{\sum_{a,b} \rho_n \frac{\partial Q_{ab}}{\partial R_{uv}} \log \left( \frac{Q_{ab}}{\rho_nn^2 q_a q_b} \right)}_{A(R)} + \underbrace{\sum_{a,b} \rho_nQ_{ab} \frac{\partial}{\partial R_{uv}} \left[ \log Q_{ab} - \log q_a - \log q_b - \log (\rho_n n^2) \right]}_{B(R)}.$$
For $B(R)$ we have
$$B(R) =\sum_{a,b} \rho_n Q_{ab} \left[\frac{1}{Q_{ab}} \frac{\partial Q_{ab}}{\partial R_{uv}}-\frac{1}{q_a} \frac{\partial q_a}{\partial R_{uv}}-\frac{1}{q_b} \frac{\partial q_b}{\partial R_{uv}}\right].$$
Since $\sum_b Q_{ab} = q_a$, we can simplify the sums 
$$\sum_{a,b} \frac{\partial Q_{ab}}{\partial R_{uv}} = \frac{\partial}{\partial R_{uv}} (\sum_{a} q_{a}), \quad \sum_{a,b} \frac{Q_{ab}}{q_a} \frac{\partial q_a}{\partial R_{uv}} = \sum_a \frac{\partial q_a}{\partial R_{uv}} \text{ and } \sum_{a,b} \frac{Q_{ab}}{q_b} \frac{\partial q_b}{\partial R_{uv}} = \sum_b \frac{\partial q_b}{\partial R_{uv}}.$$
So all these terms are identical and we obtain 
$$B(R)= \rho_n \left[ \frac{\partial}{\partial R_{uv}} (\sum_a q_a)-\frac{\partial}{\partial R_{uv}} (\sum_a q_a)-\frac{\partial}{\partial R_{uv}} (\sum_a q_a)\right] = - \rho_n\frac{\partial}{\partial R_{uv}} \sum_{a,b} Q_{ab}.$$
At $R=D$,
$$B(D)=-\rho_n \sum_{a,b} (\delta_{au} S_{vb} r_b + \delta_{bu} r_a S_{av})=- \rho_n(Sr)_v - \rho_n(Sr)_v=-2 \rho_n(Sr)_v.$$
Let us now compute $A(D)$.
Substituting $Q_{ab} = r_a S_{ab} r_b$ and $q_a = r_a (Sr)_a$ (values at $R=D$) into the logarithm 
$$\log \left( \frac{r_a S_{ab} r_b}{\rho_n n^2 (r_a (Sr)_a) (r_b (Sr)_b)} \right)=\log \left( \frac{S_{ab}}{\rho_nn^2 (Sr)_a (Sr)_b} \right).$$
We then plug in the derivative of $Q_{ab}$ 
\begin{align*}
A(D)&=\rho_n\sum_{a,b} (\delta_{au} S_{vb} r_b + \delta_{bu} r_a S_{av})\log \left( \frac{S_{ab}}{\rho_nn^2 (Sr)_a (Sr)_b} \right)\\
&=2\rho_n \sum_{a} r_a S_{av}\log \left( \frac{S_{au}}{\rho_nn^2 (Sr)_a (Sr)_u} \right)
\end{align*}
Finally, putting everything together 
$$\frac{\partial H_{S,n}}{\partial R_{uv}} \bigg|_{R=D}=2\rho_n \sum_{a} r_a S_{av}\log \left( \frac{S_{au}}{\rho_nn^2 (Sr)_a (Sr)_u} \right)-2\rho_n (Sr)_v.$$
Now we can compute the derivative of $G(\lambda) = H_{S,n}(D + \sum \lambda_{bb'} \Delta_{bb'}) = H_{S,n}(M(\lambda))$ with $M(\lambda) = \tDiag(R^\intercal \1) + \sum_{b \neq b'} \lambda_{bb'} \Delta_{bb'}$.
To differentiate with respect to a scalar parameter $\lambda_{uv}$, we use the chain rule for matrix-valued functions 
$$\frac{\partial G}{\partial \lambda_{uv}}=\sum_{i,j}\frac{\partial  H_{S,n}}{\partial M_{ij}}\times \frac{\partial M_{ij}}{\partial \lambda_{uv}}.$$
Let us look at the partial derivatives of $M$ with respect to $\lambda_{uv}$:
if $(i,j) = (u,v)$, then $\frac{\partial M_{uv}}{\partial \lambda_{uv}} = 1$,
if $(i,j) = (v,v)$, then $\frac{\partial M_{vv}}{\partial \lambda_{uv}} = -1$,
and in all other cases the derivative is zero.
Plugging these values into the chain rule sum, only two terms remain 
$$\frac{\partial G}{\partial \lambda_{uv}}=\frac{\partial  H_{S,n}}{\partial M_{uv}}-\frac{\partial  H_{S,n}}{\partial M_{vv}}.$$
At $\lambda = 0$, the matrix $M(0)$ is simply the diagonal matrix
$D = \tDiag(R^\intercal \veco)$. We therefore obtain 
\begin{align*}
\frac{\partial G}{\partial \lambda_{uv}}\bigg|_{\lambda=0} &=\frac{\partial H_{P,n}}{\partial R_{uv}}\bigg|_{R=D}-\frac{\partial H_{P,n}}{\partial R_{vv}}\bigg|_{R=D}\\
&=2 \sum_a \rho_nr_a S_{av}\left[\log \left( \frac{S_{au}}{\rho_nn^2 (Sr)_a (Sr)_u} \right)-\log \left( \frac{S_{av}}{\rho_nn^2 (Sr)_a (Sr)_{v}} \right)\right]\\
&=2 \sum_a \rho_n r_a S_{av}\log \left( \frac{S_{au} (Sr)_{v}}{S_{av} (Sr)_u} \right).
\end{align*}
Rearranging the terms to make the Kullback--Leibler divergence between two Poisson distributions appear, we obtain
$$\frac{\partial G}{\partial \lambda_{uv}}\bigg|_{\lambda=0}=-2 \rho_n \left[\sum_a r_a KL(S_{av} \| S_{au})-KL((Sr)_{v} \| (Sr)_u)
\right].$$

\section*{Acknowledgments}

I would like to sincerely thank my thesis supervisor, Catherine Matias, for her kindness and insightful guidance, which were essential to the completion of this article.

Furthermore, I acknowledge the use of artificial intelligence tools during the preparation of this manuscript. ChatGPT was utilized to enhance the overall clarity and linguistic flow of the text, while Gemini suggested the application of the Data Processing Inequality within the proof of Lemma 3. I have carefully reviewed, edited, and validated all AI-generated suggestions, and I assume full and complete responsibility for the final content of this work.

\appendixpage
\appendix
\section{Calculations of the criterion} \label{proof_calcul_Q}

\begin{lemma} \label{lemma_calcul_Q}
Let $\hat{\theta}_i$ and $\hat{S}_{ab}$ denote the maximum likelihood estimators of the parameters $\theta_i$ and $S_{ab}$ under the degree-corrected stochastic block model (see \eqref{def_est}). Then, the conditional maximum log-likelihood satisfies
\[
\max_{\boldsymbol{\theta},S}
\frac{2}{n^2}
\log \mathbb P_{\boldsymbol{\theta},S}
(\mathbf A_{n\times n} \mid \mathbf Z_n=\mathbf z_n)
=
Q(\mathbf z_n)+C,
\]
where
\begin{equation}\label{def_Q}
Q(\mathbf z_n)
=
\frac1{n^2}
\sum_{a,b}
o_{ab}(\mathbf z_n)
\log\left(
\frac{
o_{ab}(\mathbf z_n)
}{
D_a(\mathbf z_n)D_b(\mathbf z_n)
}
\right),
\end{equation}
and $C$ is a constant independent of $\mathbf z_n$.
\end{lemma}
\begin{proof}
The maximum conditional log-likelihood is given by
\begin{align*}
\widetilde{Q}(\mathbf{z}_{n}) &= \max_{\boldsymbol{\theta},S} \left\{ \frac{2}{n^2} \log \mathbb{P}_{\boldsymbol{\theta},S} (\mathbf{A}_n \mid\mathbf{Z}_n= \mathbf{z}_{n}) \right\} \\
&=\frac{1}{n^2} \left( 2 \sum_{i=1}^n d_i \log(\hat{\theta}_i) + \sum_{a,b} o_{ab}(\mathbf{z}_{n}) \log (\rho_n\hat{s}_{ab}) - n_{a}(\mathbf{z}_{n}) n_{b}(\mathbf{z}_{n}) \rho_n\hat{S}_{ab}\right) + C,
\end{align*}
where $C$ is a constant. Substituting the expressions \eqref{def_est} for $\hat{\theta}_i$ and $\rho_n \hat{S}_{ab}$ and rearranging the terms gives 
\begin{align*}
\widetilde{Q}(\mathbf{z}_{n}) &=\frac{1}{n^2} \Bigg[ 2 \sum_{i=1}^n d_i \sum_{a=1}^k \1_{z_i=a}\log\left(\frac{n_{a}(\mathbf{z}_{n})}{D_{a}(\mathbf{z}_{n})}\right) + 2 \sum_{i=1}^n d_i \log(d_i) \\
&\qquad + \sum_{a,b} o_{ab}(\mathbf{z}_{n}) \log \left(\frac{o_{ab}(\mathbf{z}_{n})}{n_{a}(\mathbf{z}_{n}) n_{b}(\mathbf{z}_{n})}\right) - \sum_{a,b} o_{ab}(\mathbf{z}_{n}) \Bigg] + C.
\end{align*}
Since $\sum_{a,b} o_{ab}(\mathbf{z}_{n})$ is the total number of edges in the graph and $\sum_i d_i \log(d_i)$ does not depend on $\mathbf{z}_n$, these terms can be absorbed into the constant. Thus, we consider 
\[
\widetilde{Q}(\mathbf{z}_{n}) =\frac{1}{n^2} \left( 2 \sum_{a=1}^k D_{a}(\mathbf{z}_{n}) \log\left(\frac{n_{a}(\mathbf{z}_{n})}{D_{a}(\mathbf{z}_{n})}\right) + \sum_{a,b} o_{ab}(\mathbf{z}_{n}) \log \left(\frac{o_{ab}(\mathbf{z}_{n})}{n_{a}(\mathbf{z}_{n}) n_{b}(\mathbf{z}_{n})}\right) \right) + C.
\]
Using the relation $D_a(\mathbf{z}_n)=\sum_b o_{ab}(\mathbf{z}_n)$, we first obtain 
\[
\sum_{a=1}^k D_{a}(\mathbf{z}_{n}) \log\left(\frac{n_{a}(\mathbf{z}_{n})}{D_{a}(\mathbf{z}_{n})}\right) = \frac{1}{2} \sum_{a,b} o_{ab}(\mathbf{z}_{n}) \log \left( \frac{n_a(\mathbf{z}_{n}) n_b(\mathbf{z}_{n}) }{D_a(\mathbf{z}_{n}) D_b(\mathbf{z}_{n}) } \right).
\]
Then we have
\[
\widetilde{Q}(\mathbf{z}_{n}) =\frac{1}{n^2}  \sum_{a,b} o_{ab}(\mathbf{z}_{n}) \left(  \log \left( \frac{n_a(\mathbf{z}_{n}) n_b(\mathbf{z}_{n}) }{D_a(\mathbf{z}_{n}) D_b(\mathbf{z}_{n}) } \right) +\log \left(\frac{o_{ab}(\mathbf{z}_{n})}{n_{a}(\mathbf{z}_{n}) n_{b}(\mathbf{z}_{n})}\right) \right) + C.
\]
We further simplify the likelihood modularity to 
\begin{equation*}
Q(\mathbf{z}_{n}) = \frac{1}{n^2} \sum_{a,b} o_{ab}(\mathbf{z}_{n}) \log \left( \frac{ o_{ab}(\mathbf{z}_{n}) }{D_a(\mathbf{z}_{n}) D_b(\mathbf{z}_{n}) } \right).
\end{equation*}
\end{proof}

\section{Proof of Lemma \ref{lemma_C_positif}} \label{proof_C_positif}
We define the vector $\lambda = S\pi$, so that for every block $b \in [k]$,

\[
\lambda_b = [S\pi]_b = \sum_{a=1}^k \pi_a S_{ab}.
\]
We therefore have
\[
C(\pi, S) = \min_{b \neq b'} \sup_{t \in (0,1]} f_{bb'}(t),
\]
where, for a pair of distinct blocks $b \neq b'$, the functional $f_{bb'}(t)$ is given by
\[
f_{bb'}(t)
=
\sum_{a=1}^k \pi_a H_t(S_{ab'} \| S_{ab})
-
t\, KL(\lambda_{b'} \| \lambda_b).
\]
To prove that $C(\pi,S) > 0$, it is sufficient to show that for every pair $b \neq b'$,
\[
\sup_{t \in (0,1]} f_{bb'}(t) > 0.
\]
For readability, we write $f(t)$ instead of $f_{bb'}(t)$. Expanding the first term and using the identity
\[
\sum_a \pi_a S_{ab} = \lambda_b,
\]
we obtain
\[
\sum_{a=1}^k \pi_a H_t(S_{ab'} \| S_{ab})
=
(1-t)\lambda_{b'}
+
t\lambda_b
-
\sum_{a=1}^k \pi_a S_{ab'}^{1-t} S_{ab}^t.
\]
For the second term, we have

\[
t\, KL(\lambda_{b'} \| \lambda_b)
=
t\lambda_{b'} \log \frac{\lambda_{b'}}{\lambda_b}
+
t\lambda_b
-
t\lambda_{b'}.
\]
Combining both expressions, the linear terms in $\lambda_b$ and $\lambda_{b'}$ cancel out, yielding
\[
f(t)
=
\lambda_{b'}
-
\sum_{a=1}^k \pi_a S_{ab'}^{1-t} S_{ab}^t
-
t\lambda_{b'} \log \frac{\lambda_{b'}}{\lambda_b}.
\]
Evaluating this expression at $t=0$, we immediately obtain
\[
f(0)
=
\lambda_{b'}
-
\sum_{a=1}^k \pi_a S_{ab'}
-
0
=
\lambda_{b'}
-
\lambda_{b'}
=
0.
\]
We now differentiate $f(t)$ with respect to $t$. Using the identity
\[
\frac{d}{dt}(u^{1-t}v^t)
=
u^{1-t}v^t \log\frac{v}{u},
\]
we find
\[
f'(t)
=
\sum_{a=1}^k
\pi_a
S_{ab'}^{1-t} S_{ab}^t
\log\frac{S_{ab'}}{S_{ab}}
-
\lambda_{b'}
\log\frac{\lambda_{b'}}{\lambda_b}.
\]
Evaluating this derivative at $t=0$ gives
\[
f'(0)
=
\sum_{a=1}^k
\pi_a S_{ab'}
\log\frac{S_{ab'}}{S_{ab}}
-
\lambda_{b'}
\log\frac{\lambda_{b'}}{\lambda_b}.
\]
A direct application of the log-sum inequality yields
\[
\sum_{a=1}^k
\pi_a S_{ab'}
\log\frac{S_{ab'}}{S_{ab}}
\geq
\lambda_{b'}
\log\frac{\lambda_{b'}}{\lambda_b},
\]
which immediately implies
\[
f'(0) \geq 0.
\]
The log-sum inequality becomes an equality if and only if the ratio of the terms is constant for all $a$. That is, if there exists a constant $c>0$ such that
\[
\frac{\pi_a S_{ab'}}{\pi_a S_{ab}}
=
c
\quad\Longrightarrow\quad
S_{ab'}
=
c\, S_{ab}
\qquad
\forall a \in \{1,\dots,k\}.
\]
In matrix form, this condition means that column $b'$ of the matrix $S$ is a scalar multiple of column $b$. However, by assumption, $S$ has full rank. Since we are considering the case $b \neq b'$, it is therefore impossible for column $b'$ to be collinear with column $b$. Hence the equality condition is violated, which guarantees that the inequality is strict:
\[
f'(0) > 0.
\]
We have therefore shown that $f(0)=0$, $f(t)$ is continuously differentiable on $[0,1]$ and $f'(0)>0$. By a Taylor expansion around $0$, there exists $\varepsilon>0$ such that for all $t \in (0,\varepsilon)$,
\[
f(t) > 0.
\]
Consequently, the supremum over the interval $(0,1]$ is strictly positive for this pair:
\[
\sup_{t \in (0,1]} f_{bb'}(t) > 0
\qquad
\forall b \neq b'.
\]
Since the number of blocks $k$ is finite, the set of pairs $\{(b,b') : b \neq b'\}$ 
is finite as well. The minimum of finitely many strictly positive values is itself strictly positive, and we finally conclude that, when $S$ is full rank,
\[
C(\pi,S) > 0.
\]

\end{document}